\documentclass[10pt]{amsart} 

\usepackage[psamsfonts]{amssymb} 
\usepackage{amsfonts,amsmath,mathdots} 
\usepackage{graphicx, color}
\usepackage{enumerate}
\usepackage{url}
\usepackage{tikz-cd}
\usepackage[hidelinks]{hyperref}

\usepackage{amsthm, amssymb} 

\usepackage{enumitem}

\usepackage{titletoc}

\usepackage[all]{xy}

\newtheorem{thmA}{Theorem}

\newtheorem{theorem}{Theorem}[section]
\newtheorem{prop}[theorem]{Proposition}
\newtheorem{proposition}[theorem]{Proposition}
\newtheorem{lemma}[theorem]{Lemma}
\newtheorem{corollary}[theorem] {Corollary}

\newtheorem*{claim*}{Claim}

\theoremstyle{remark}
\newtheorem{remark}[theorem]{Remark}

\newtheorem{warning}[theorem]{Warning}

\theoremstyle{definition}
\newtheorem{definition}[theorem]{Definition}
\newtheorem{notation}[theorem]{Notation}

\def\calg{\mathcal{G}}

\def\F{\mathcal{F}}

\def\Z{\mathbb Z}

\def\N{\mathbb N}
\def\G{\mathcal{G}}

\def\hom{{\rm{Hom}}}

\def\<{\langle}
\def\>{\rangle}

\newcommand{\st}{\mathrm{st}}
\newcommand{\lk}{\mathrm{lk}}

\newcommand{\homcd}{\hom_{\textnormal{c}}}
\newcommand{\rar}{\rightarrow}
\def\co{\colon}
\def\ostar{\mathring\st}
\newcommand{\coker}{\textnormal{coker}}
\newcommand{\id}{\textnormal{id}}
\newcommand{\colim}{\mathop{\textnormal{colim}}}

\newcommand{\under}{\triangleleft}
\newcommand{\dcap}{\mathfrak{c}}

\newcommand{\rmod}{R\tn{-}\mathbf{Mod}}
\newcommand{\modr}{\mathbf{Mod}\tn{-}R}
\newcommand{\rmodr}{R\tn{-}\mathbf{Mod}\tn{-}R}

\newcommand{\hres}[1]{h_*|_{#1}}

\newcommand{\vc}{^\textnormal{vc}}
\newcommand{\lf}{\textnormal{lf}}

\newcommand{\tn}[1]{\textnormal{#1}}

\renewcommand{\geq}{\geqslant}		
\renewcommand{\leq}{\leqslant}

\newcommand{\poset}{X_{\bullet}}	

\newcommand{\fund}[1]{\bar{#1}}		

\newcommand{\chain}{a}				
\newcommand{\localchain}{a}			
\newcommand{\shchain}{\eta}			

\newcommand{\cochain}{\phi}			
\newcommand{\intcochain}{\psi}		
\newcommand{\shcochain}{\phi}		

\newcommand{\bigsimplex}{\alpha}	
\newcommand{\smallsimplex}{\sigma}	
\newcommand{\bigsimplexone}{\beta}

\newcommand{\smallsimplexone}{\tau}
\newcommand{\smallsimplextwo}{\rho}

\newcommand{\ind}[2]{\mathfrak{o}_{#1}(#2)}	

\newcommand{\sg}[1]{\pm_{#1}}

\DeclareMathOperator{\im}{im}

\title[Duality for CM Complexes through Combinatorial Sheaves]{Duality for Cohen--Macaulay Complexes through Combinatorial Sheaves}

\author{Richard D. Wade and Thomas A. Wasserman}   

\begin{document}

\begin{abstract} 
We prove a duality theorem for Cohen--Macaulay simplicial complexes, working over an arbitrary ring. This is a generalisation of Poincar\'e Duality, framed in the language of combinatorial sheaves. Our treatment is self-contained and accessible for readers with a working knowledge of simplicial complexes and (co)homology. The main motivation is a link with Bieri-Eckmann duality for discrete groups, which is explored in a companion paper.
\end{abstract}

\address{}
\email{}

\maketitle
\setcounter{tocdepth}{2}

\section{Introduction}
A space $X$ is \emph{Cohen--Macaulay} (CM) of dimension $n$ if the reduced homology $\tilde{H}_*(X)$ and local homology groups $H_{*}(X,X-x)$ for $x\in X$ are trivial unless $*=n$. CM spaces appear throughout combinatorial, algebraic and geometric topology. If $X$ satisfies just the local homology condition, $X$ is called locally CM. Locally CM spaces can be thought of as homology manifolds with mild singularities. Akin to Poincar\'e duality, a locally CM and locally finite simplicial complex $X$ admits an isomorphism induced by a cap product between cohomology and homology with appropriate coefficients. We call this isomorphism \emph{CM duality}. 

\subsection{Results}

\subsubsection{CM Duality}
A \emph{combinatorial sheaf} on a simplicial complex $X$ is a functor out of the poset of simplices (also denoted $X$) in $X$, ordered by inclusion. A \emph{combinatorial cosheaf} is functor out of $X^\text{op}$. Taking local homology at simplices defines a (combinatorial) sheaf $h_*$ on $X$, whereas local cohomology defines a cosheaf $h^*$.
\begin{thmA}[CM Duality]\label{introdualitythm}
	Let $X$ be a locally finite, locally Cohen--Macaulay simplicial complex of dimension $n$. Then there is a fundamental class $[X]\in H^{\lf}_n(X;h^*)$ in locally finite homology, such that the cap product induces isomorphisms
	$$
	[X]\cap-\colon H_c^p(X;h_*)\xrightarrow{\cong} H_{n-p}(X;\Z)
	$$
	and
	$$
	[X]\cap-\colon H_c^p(X;\Z)\xrightarrow{\cong} H_{n-p}(X;h^*),
	$$	
	where $H^*_c$ denotes compactly supported cohomology.
\end{thmA}
A more comprehensive statement is given in Theorem~\ref{dualitytheorem}. There we work over an arbitrary ring $R$ and relative to subcomplexes, and give corresponding versions for non-compactly-supported cohomology and locally finite homology. We additionally explain how to obtain CM duality with coefficients in arbitrary (co)sheaves.

\subsubsection{Cap product}
In Section~\ref{capproductsection} we give an explicit construction of the cap product between combinatorial cosheaf homology and sheaf cohomology as a chain map out of a double complex. At the chain level this involves some choices that depend on an orientation of $X$;  we prove that the induced map on (co)homology is independent of these choices.

\subsubsection{Naturality}
CM duality is natural in the following sense:

\begin{thmA}\label{intronatthm}
	Let $X$ and $Y$ be locally CM simplicial complexes and let $f\colon X \rar Y$ be a local homeomorphism that restricts to a simplicial isomorphism on stars. Then the duality isomorphism from Theorem~\ref{introdualitythm} is natural with respect to the maps induced by $f$ in (co)homology.
\end{thmA}
This is a part of Theorem~\ref{naturalityofduality}. When a group $G$ acts on a locally CM complex $X$ the (co)homology groups that appear in CM duality are $G$-modules. Theorem~\ref{intronatthm} implies that the isomorphisms from Theorem~\ref{introdualitythm} are isomorphisms of $G$-modules.

\subsubsection{Cosheaf homology in degree 0} If $\mathcal{F}$ is a sheaf then $H^0(X;\mathcal{F})$ has a geometric description as the group of global sections $\Gamma(\mathcal{F})$ of $\mathcal{F}$. In this paper we examine cosheaf homology in degree 0, for the particular case of the local cohomology cosheaf $h^*$.  With some restrictions on the local homology sheaf, we show that $H_0(X;h^*)$ is a subgroup of the dual to the global section group $\Gamma(h_*)$ of the local homology sheaf.

Recall that an inverse system of abelian groups is called \emph{semi-stable} if the images of the maps in the system stabilise (sometimes this is called the \emph{Mittag-Leffler condition}). A homomorphism $f: \Gamma(\mathcal{F}) \to \mathbb{Z}$ in the dual of $\Gamma(\mathcal{F})$ is \emph{compactly determined } if there exists a compact set $K$ such that $f(s)=0$ for every section $s$ that is trivial on $K$.
\begin{thmA}\label{introdualsheaf}
	Let $X$ be a locally CM simplicial complex, and let $\{K_i\}_{i\in \N}$ be a filtration of $X$ by finite subcomplexes. Then the inverse system $\{\Gamma(\hres{K_i})\}_{i\in \N}$ is semi-stable if and only if $\Gamma(h_*)$ is a free abelian group. In this case,
	\begin{equation*}
		H_0(X;h^*)\xrightarrow{\cong} \homcd(\Gamma(h_*),\Z),	
	\end{equation*}
	where $\homcd$ denotes the group of compactly determined homomorphisms.
\end{thmA}
This follows from Corollary~\ref{c:semistable} in the text, which is phrased over PIDs rather than $\mathbb{Z}$.

\subsection{A brief history: CM complexes, duality, and combinatorial sheaves}

\subsubsection{Cohen--Macaulay complexes}
The name \emph{Cohen--Macaulay complex} has its origin in a theorem of Reisner \cite{Reisner}, who showed that a finite simplicial complex is CM if and only if its \emph{face ring} (also called the \emph{Stanley--Reisner ring}) is  Cohen--Macaulay as a ring (although the CM definition for spaces given above is a reformulation due to Munkres \cite{Munkres1984}). This equivalence was a key component in Stanley's proof of the Upper Bound Conjecture on triangulations of spheres (\cite{Stanley}, see also \cite{Bjoerner2016, Hochster} for surveys), and is an important bridge between commutative algebra and combinatorial topology. In the study of Stanley-Reisner rings dualities related to the ones discussed in this paper have been observed \cite{Yanagawa2003,Reiner2001} for finite complexes. In \cite{Varbaro2024} these are used to prove duality theorems for (homology) manifolds.

A few years after the papers of Stanley and Reisner, Quillen proved that the geometric realizations of certain subgroup posets satisfy a stronger, homotopical version of the CM condition \cite{Quillen1978}. In part due to Quillen's influence,  the homotopical version of the CM property is more common in papers in algebraic/geometric topology.

\subsubsection{CM duality and combinatorial sheaves} In algebraic topology, CM spaces appear without being explicitly named in McCrory's study of the \emph{Zeeman spectral sequence} \cite{McCrory1979}, and in Bredon's book on sheaf theory \cite{Bredon2012} as \emph{weak homology manifolds}. Zeeman introduced his spectral sequence as a generalisation of Poincar\'e duality \cite{ZeemanDihomologyIII}. Zeeman also introduced combinatorial sheaves under the name of stacks \cite{Zeeman1962}. McCrory \cite{McCrory1979} took Zeeman's work further: he used Zeeman's spectral sequence to obtain results for relative (co)homology and gave a geometric interpretation of the groups on the $E_2$-page. He also abstractly identifies \cite[Section 6.3]{McCrory1979} a cap product on homology manifolds with the edge homomorphisms of this spectral sequence.

These duality results for CM complexes are not well-known. One reason for this might be the success of \emph{intersection homology} (pioneered by Goresky and MacPherson \cite{Goresky1980}) in studying (co)homology of singular spaces. Intersection homology is most often applied to \emph{pseudomanifolds}, which form a generalisation of manifolds orthogonal to CM spaces.

More recently, in \cite{Ginzburg1994} combinatorial sheaves are used to work with sheaves over operads. The cellular cousins to combinatorial sheaves have been used by Curry \cite{Curry2014} in the context of Topological Data Analysis.

\subsubsection{Verdier duality}
In the same year that Zeeman published his generalisation of Poincar\'e duality, Verdier \cite{Verdier1963} announced what is now called \emph{Verdier duality}. While our proof of CM duality does not directly use this, CM duality can be seen as a consequence of Verdier duality and our strategy can be viewed as providing an explicit counterpart to abstract (derived) considerations. For context, we briefly sketch these considerations here.

It is well-known that Verdier duality for sheaves gives rise to Poincar\'e duality (see for example \cite{Iversensheaf}). Lurie \cite[Section 5.5.5]{Lurie2017} recasts Verdier duality as the statement that taking compactly supported sections induces, after passing to derived categories, an equivalence between the category of sheaves and the category of cosheaves on a (nice) space. For a finite-dimensional locally finite Hausdorff space $X$, the equivalence takes the the constant cosheaf $\Z$ to a sheaf of chain complexes $c_*$ which computes the local homology sheaf $h_*$. This gives rise to a \emph{hypercohomology spectral sequence}
\begin{equation}
	E^2_{-p,q}=H^p(X;h_q) \Rightarrow H_{q-p}(X; \Z). \label{hypercohspectseq}
\end{equation}
This spectral sequence is closely related to Zeeman's spectral sequence. When $X$ is locally CM the spectral sequence collapses to give the first isomorphism from Theorem~\ref{introdualitythm}. However, this proof makes no reference to a cap product or fundamental class, and inherently requires working in the derived setting.

\subsubsection{Duality groups} A major motivation for this paper comes from \emph{duality groups}. Bieri and Eckmann \cite{Bieri1973} defined a duality group of dimension $n$ to be a group $G$ that has a fundamental class $c\in H_n(G;\mathcal{D})$ so that
$$
H^p(G;M)\xrightarrow{c \cap -} H_{n-p}(G;\mathcal{D} \otimes M)
$$
is an isomorphism for all $G$-modules $M$ and some fixed $G$-module $\mathcal{D}$ called the \emph{dualizing module}. The reader will notice the similarity between this and CM Duality. This similarity is not coincidental; every duality group with finite classifying space known to the authors acts freely and cocompactly on a contractible CM space. This connection is explored further in \cite{WadeWass}, where we use Theorem~\ref{introdualsheaf} to give new descriptions of dualizing modules.

\subsection{A sketch of proof of CM duality}
The CM Duality Theorem (Theorem~\ref{dualitytheorem}) contains eight different duality isomorphisms for a locally finite oriented locally Cohen--Macaulay simplicial complex $X$ of dimension $n$ and a full subcomplex $L$. We now sketch a proof of the duality isomorphism between the compactly supported cohomology of the local homology sheaf $H^p_c(L;h_n|_L)$ restricted to $L$ and the homology $H_{n-p}(X,L\vc)$ of $X$ relative to the \emph{vertex complement} $L\vc$ which is spanned by the vertices of $X$ that are not in $L$. That is, we will sketch how to show that
$$
[X]\cap -\colon H^p_c(L;h_n|_L) \xrightarrow{\cong} H_{n-p}(X,L\vc; \Z).
$$

Our strategy of proof is closely related to the hypercohomology spectral sequence from Equation~\eqref{hypercohspectseq}. We use the spectral sequences coming from the row and column filtrations of the double complex $D^\bullet_\bullet(L)$ given by
$$
D^l_k(L)=\bigoplus_{\smallsimplex \in L_l} C_k(\smallsimplex).
$$
Here $L_l$ denotes the set of $l$-simplices of $L$ and $C_\bullet(\smallsimplex)$ is the simplicial chain complex that computes the local homology of $X$ at $\smallsimplex$. This double complex can always (without the local CM condition) be augmented with a map 
$$
C_{\bullet}(X,L\vc;\Z)\xrightarrow{\epsilon} D^\bullet_\bullet(L)
$$ 
to a double complex with exact rows (Proposition~\ref{doubleexactrows}). This gives the isomorphism
$$
\epsilon_*\colon H_{k}(X,L\vc) \xrightarrow{\cong} H_k(\mathbf{Tot}D^\bullet_\bullet(L)),
$$
between $H_{k}(X,L\vc)$ and the homology of the total complex. When $X$ is locally CM, the column spectral sequence that converges to $H_*(\mathbf{Tot}D^\bullet_\bullet(L))$ collapses, and we get
\begin{equation}\label{firstdualityiso}
	\epsilon_*\colon H_{n-p}(X,L\vc)\xrightarrow{\cong} H_c^p(L;h_n|_L),	
\end{equation}
where $H^p_c(L;h_n|_L)$ is the group in position $(p,n)$ on the $E_2$ page.

We then show that this isomorphism is inverse to the cap product with a fundamental class. We introduce the CM fundamental class $[X]\in H_n(X;h^n)$ in Section~\ref{fundclass}. The cap product with this class gives a map
$$
[X]\cap - \colon H^p(L;h_n|_L) \rar H_{n-p}(X,L\vc;\Z).
$$ 
To show this cap product agrees with an inverse $\dcap$ for $\epsilon_*$ we prove a beefed up version of the snake lemma (Python Proposition~\ref{pythonlemma}) homologouing a cycle in $D^l_k(L)$ to a cycle in $D^{0}_{k-l}(L)$ by doing iterative \emph{last-vertex lifts}. To our knowledge this procedure is completely new, and results in an explicit chain level formula for $\dcap$ that manifestly agrees with the chain level expression for $[X]\cap-$. This shows that the isomorphism in Equation~\eqref{firstdualityiso} is induced by the cap product.

\subsection{Outline} 
\subsubsection{Setup} The first half of Section~\ref{setupsection} is devoted to giving the necessary background for this paper. We introduce combinatorial sheaves and cosheaves, and then discuss the local homology sheaf and local cohomology cosheaf. After this we discuss the cap product in this setting, and show that it is a chain map and independent of any choices made.

\subsubsection{CM complexes} In Section~\ref{CMsection} we review Cohen--Macaulay complexes. We prove the characterisation of the degree zero homology of the local cohomology cosheaf of a CM complex from Theorem~\ref{introdualsheaf} here.

\subsubsection{CM Duality}
In Section~\ref{cmdualitysect} we prove CM duality (Theorem~\ref{dualitytheorem}). 

\subsubsection{Functoriality}
In Section~\ref{functorialitysection} we discuss naturality for CM duality, with respect to a class of maps we refer to as star-local homeomorphisms. These are simplicial local homeomorphisms that additionally behave well on stars of simplices.

\subsection*{Acknowledgments} We thank Vidit Nanda for many helpful discussions, particularly when this project was in its infancy. We thank Andr\'e Henriques for pointing us to Lurie's work on Verdier duality and other helpful discussions. We thank Martin Palmer and Arthur Souli\'e for inquiring about the validity of our results for non-commutative rings. We thank the referee for giving valuable comments. Both authors are supported by the Royal Society of Great Britain.

\section*{Table of contents}
\startcontents
\printcontents{ }{1}{}

\section{Combinatorial Sheaves and Cosheaves}\label{setupsection}
In this section we introduce combinatorial sheaves and cosheaves in Section~\ref{introsimplsheaves} and discuss their cohomology and homology in Section~\ref{sheafcohomologysection}. We then introduce the local homology sheaf and local cohomology cosheaf in Section~\ref{lochomshsect}. In Section~\ref{capproductsection} we discuss the cap product in this setting.

\subsection{Introduction to combinatorial sheaves}\label{introsimplsheaves}
We fix our notation for simplicial complexes in \textsection\ref{simpcomplsect}. We then introduce the main tools of this paper, combinatorial sheaves and cosheaves, in \textsection\ref{simplshsect}, and discuss their sections in \textsection\ref{sectionsection}.

\subsubsection{Simplicial complexes}\label{simpcomplsect}
Throughout we fix a locally finite oriented simplicial complex $X$. Recall that an \emph{orientation} of a simplicial complex is an ordering of the vertices in each simplex such that face inclusions are order preserving. We will denote by $\poset$ the poset of simplices of $X$ partially ordered by inclusion. 

\begin{notation}\label{simplexnotation}
	We will denote the set of $k$-simplices of $X$ by $X_k$. For $0\leq j \leq k\leq l $ and simplices $\smallsimplex\in X_k$ and $\smallsimplexone\in X_l$ we will denote by
	\begin{itemize}
		\item $\smallsimplex \leq \smallsimplexone$ the simplex $\smallsimplex$ being a face of $\smallsimplexone$ of any codimension;
		\item $\smallsimplex < \smallsimplexone$ the simplex $\smallsimplex$ being a face of $\smallsimplexone$ of any positive codimension;
		\item $\smallsimplex_j\in X_0$ the $j$th vertex of $\smallsimplex$;
		\item $\smallsimplex_{\langle j \rangle}\in X_{k-1}$ the $j$th face of $\smallsimplex$, spanned by the vertices of $\smallsimplex$ excepting $\smallsimplex_j$;
		\item $\smallsimplex_{\leq j}\in X_j$ the $j$-front face of $\smallsimplex$, spanned by the vertices $\sigma_0, \sigma_1,\ldots, \sigma_j$;
		\item $\smallsimplex_{\geq j}\in X_{k-j}$ the $k-j$-back face of $\smallsimplex$, spanned by the vertices $\sigma_j,\sigma_{j+1} \ldots, \sigma_k$;
		\item $\smallsimplex \star \smallsimplexone$ the \emph{(internal) join}\footnote{This internal notion of join is different from the usual join in that we require the vertices to span a simplex \emph{in} $X$.}, the simplex spanned by the union of the vertices of $\smallsimplex$ and $\smallsimplexone$ (this is possibly degenerate, and is defined to be empty if the vertices do not span a simplex in $X$);
		\item $\st \smallsimplex$ the \emph{star of $\smallsimplex$}, the full subcomplex on those simplices that meet $\smallsimplex$;
		\item $\ostar \smallsimplex$ the \emph{open star of $\smallsimplex$}, the union of the interiors of the simplices that meet $\smallsimplex$;
		\item $\lk \smallsimplex$ the \emph{link of $\smallsimplex$}, consisting of those simplices $\smallsimplexone$ not intersecting $\smallsimplex$ for which $\smallsimplex\star \smallsimplexone$ is not empty. Equivalently, $\lk \smallsimplex=\st \smallsimplex - \ostar \smallsimplex$.
	\end{itemize}
\end{notation}

By a subcomplex $L$ we will always mean a full subcomplex, meaning that whenever vertices in $L$ span a simplex in $X$ that simplex is also in $L$. This is for simplicity: any subcomplex is homeomorphic to a full subcomplex of the barycentric subdivision of $X$.

\subsubsection{Modules over rings}\label{s:modules}
Throughout this paper we fix a ring $R$, note that we allow for $R$ to be noncommutative and will be careful in distinguishing left and right modules. We will denote the categories of left, right and bimodules over $R$ by $\rmod$, $\modr$ and $\rmodr$, respectively. The symbol $\otimes$ will denote the tensor product $\otimes_R$ over $R$, and $\hom$ will denote $R$-linear maps in $\rmod$ or $\modr$. If both $M$ and $N$ are bimodules, then $\hom_{\rmod}(M,N)$ carries a bimodule structure given by 
\begin{equation*}
	(r\cdot f \cdot s)(m)= f(mr)s,
\end{equation*}
for $r,s \in R$ and $m \in M$. More generally, the above gives $\hom_{\rmod}(M,N)$ a left action (by $r \in R$) if $M$ is a bimodule, and a right action (by $s \in R$) if $N$ is a bimodule. If $M$ and $N$ are only left $R$-modules and $R$ is noncommutative then there is no natural action of $R$ on $\hom_{\rmod}(M,N)$.

\subsubsection{Combinatorial sheaves and cosheaves}\label{simplshsect}
A \emph{combinatorial sheaf (of left $R$-modules)} $\mathcal{F}$ on $X$ is a covariant functor from the poset $\poset$ to the category of left $R$-modules $\rmod$.\footnote{One can show that the category of combinatorial sheaves on a simplicial complex $X$ is equivalent to the category of constructible sheaves subordinate to the natural stratification induced by the simplicial structure.} A \emph{combinatorial cosheaf (of right $R$-modules)} $\mathcal{G}$ on $X$ is a contravariant functor from $\poset$ to right $R$-modules $\modr$. In the rest of this text we will refer to these as sheaves and cosheaves whenever there is no confusion possible with the classical notions of sheaf and cosheaf. Our sheaves and cosheaves will often be valued in bimodules. However, the constructions done in this paper will only rely on a left module structure for sheaves and a right module structure for cosheaves.

To expand on the above definitions, a sheaf $\mathcal{F}$ assigns a module $\mathcal{F}(\smallsimplex)$ to each simplex of $X$ and if $\smallsimplex < \smallsimplexone$ there is an associated morphism  \[ \mathcal{F}(\smallsimplex < \smallsimplexone) \co \mathcal{F}(\smallsimplex) \to \mathcal{F}(\smallsimplexone). \] Furthermore these maps are compatible in the obvious way: if $\smallsimplex < \smallsimplexone < \nu$ then $\mathcal{F}(\smallsimplex < \nu)= \mathcal{F}(\smallsimplexone < \nu) \circ \F(\smallsimplex < \smallsimplexone)$. A cosheaf is similar, however the maps go down rather than up (with respect to the dimension of the simplices): if $\smallsimplex < \smallsimplexone$ we have a morphism $\G(\smallsimplexone > \smallsimplex)$ from $\G(\smallsimplexone)$ to $\G(\smallsimplex)$ and again these morphisms are compatible in the obvious way.

A sheaf (respectively cosheaf) is said to be \emph{locally constant} if every associated morphism $\mathcal{F}(\smallsimplex < \smallsimplexone)$ (respectively $\mathcal{G}(\smallsimplexone > \smallsimplex))$ is an isomorphism. For every module $M$ there is a \emph{constant sheaf} given by taking $\mathcal{F}(\sigma)=M$ and taking every associated morphism to be the identity. Not every locally constant sheaf is isomorphic to a constant sheaf. The constant cosheaf on a module is defined in a similar way. With rings (particularly when taking homology/cohomology) we will abuse notation by using $R$ to also denote the constant sheaf or cosheaf on a space $X$ given by $R$: this should be clear with context.

\subsubsection{Sections and cosections of sheaves and cosheaves}\label{sectionsection}
Given a (combinatorial) sheaf $\mathcal{F}$ and a subcomplex $L\subset X$ we have the \emph{restriction} $\mathcal{F}|_L$ of $\mathcal{F}$ to $L$ given by precomposing $\mathcal{F}$ with the inclusion of $L$ into $X$. The module of \emph{sections $\Gamma(\mathcal{F},L)$ on $L$} is the left $R$-module of natural transformations between the constant sheaf $R$ and $\mathcal{F}|_L$:
$$
\Gamma(\mathcal{F},L)= \{s\colon R \Rightarrow \mathcal{F}|_L\}.
$$
Unpacking this, we see that a section $s$ is given by a choice of element $s_\smallsimplex\in \mathcal{F}(\smallsimplex)\cong \hom(R,\mathcal{F}(\smallsimplex))$ for each simplex $\smallsimplex$ of $L$, subject to the condition that for $\smallsimplex< \smallsimplexone$ we have (here, and on many subsequent occasions, we will drop the parentheses around arguments to prevent a wall of brackets)
$$
\mathcal{F}(\smallsimplex<\smallsimplexone) s_\smallsimplex= s_\smallsimplexone.
$$
Note that a section is completely determined by its values on the vertices of $L$. 
The left $R$-module structure on $\Gamma(\mathcal{F},L)$ is induced by the left $R$-module structure on $\hom_{\rmod}(R,\mathcal{F}(\smallsimplex))$ cf. \textsection\ref{s:modules}.

\begin{warning}
	One might expect that $\Gamma(\mathcal{F},\smallsimplex)$ is the same as $\mathcal{F}(\smallsimplex)$, but this is generally false unless $\smallsimplex$ is a vertex. In terms of classical sheaves $\mathcal{F}(\smallsimplex)$ can be thought of as the sections over the open star $\ostar \smallsimplex$, while $\Gamma(\mathcal{F},\smallsimplex)$ is the module of sections on the union of the open stars of the vertices of $\smallsimplex$.
\end{warning}

\subsection{Sheaf cohomology and cosheaf homology}\label{sheafcohomologysection}
In this section we define sheaf cohomology and cosheaf homology via explicit chain complexes in \textsection\ref{homcohomsimplsh}, and define relative versions in \textsection\ref{relativecohsect}. The definitions given will depend on a choice of orientation of $X$; we discuss the orientation independence of the resulting (co)homologies in \textsection\ref{orindepcoh}.

\subsubsection{Homology and cohomology for combinatorial (co)sheaves}\label{homcohomsimplsh}
Recall that $X_k$ denotes the set of $k$-simplices of $X$ and that $X$ is assumed to be oriented. Given a combinatorial sheaf $\F$ there is a cochain complex $C^*(X,\F)$ of left $R$-modules where 
\begin{equation}\label{sheafcohomchains}
	C^k(X,\F)= \prod_{\smallsimplex \in X_k} \F(\smallsimplex).
\end{equation}
The boundary operator $\delta$ is defined similarly to usual cohomology: if $\shcochain \in C^k(X,\F)$ and $\smallsimplex$ is a $k+1$ simplex then  
\begin{equation}\label{sheafboundary}
	(\delta\shcochain)_\smallsimplex = \sum_{i=0}^{k+1} (-1)^i\F(\smallsimplex_i<\smallsimplex)\shcochain_{\smallsimplex_{\langle i\rangle}}
\end{equation}
is the value of $\delta\shcochain$ at $\smallsimplex$, where $\shcochain_{\smallsimplex_{\langle i\rangle}}$ is the value of $\shcochain$ on the $i$th face of $\smallsimplex$. We define the resulting cohomology left modules $H^*(X,\mathcal{F})$ to be the \emph{sheaf cohomology of $\mathcal{F}$ on $X$}. We define the \emph{compactly supported cohomology of $\mathcal{F}$ on $X$} (denoted $H_c^*(X,\mathcal{F})$) to be the cohomology obtained by replacing the direct product with a direct sum in the definition of the chain complex (i.e. we require the cochains to be nonzero on only finitely many simplices). The local finiteness of $X$ ensures that the boundary map remains well-defined.

The (compactly supported) cohomology in degree zero of a sheaf $\mathcal{F}$ on a simplicial complex $X$ is exactly its module of (compactly supported) global sections $\Gamma(\mathcal{F})$ (or $\Gamma_c(\mathcal{F})$). One way to see this is to note that the cochain complexes computing (compactly supported) sheaf cohomology admit an augmentation:
\begin{align}\label{augmentedsheafcoh}
	0&\rar \Gamma(\mathcal{F}) \rar C^0(X, \mathcal{F}) \rar C^1(X;\mathcal{F})\rar \dots	\\
	0&\rar \Gamma_c(\mathcal{F}) \rar C_c^0(X, \mathcal{F}) \rar C_c^1(X;\mathcal{F})\rar \dots.\label{augmentedcompact}
\end{align}
Both maps take a section $s\in \Gamma(\mathcal{F})$ (or $\Gamma_c(\mathcal{F})$) to $\sum_{v \in X} s_v$, where $s_v \in \mathcal{F}(v)$ is the value of the section at $v$. 

\emph{Cosheaf homology of a cosheaf $ \mathcal{G}$} is defined as the homology of the chain complex $C_*(X,\G)$ of right $R$-modules with \[C_n(X,\G)=\bigoplus_{\smallsimplex \in X_n} \G(\smallsimplex). \] The boundary operator $\partial$ is defined on $\chain \in C_n(X,\G)$ by 
\begin{equation}\label{coshcohdiff}
	\partial\chain=\sum_{\smallsimplex \in X_n}\sum_{i=0}^n (-1)^i\G(\smallsimplex > \smallsimplex_i)\chain_\smallsimplex.
\end{equation}
The homology of this chain complex is denoted $H_*(X, \calg)$, note that this is a right $R$-module. \emph{Locally finite homology} $H^{\lf}_*(X,\G)$ is defined by replacing the sum with a direct product in the definition of the chain complex (i.e. allowing for infinite chains). This can be seen as a simplicial version of Borel--Moore homology.\footnote{We remind the reader that we assume our simplicial complexes to be locally finite, dropping this assumption would lead to issues when defining the differential in locally finite homology.}

\begin{remark}
	For sheaves and cosheaves with values in $R$-bimodules, the differentials are bimodule maps, and the resulting (co)homologies are bimodules.
\end{remark}

\begin{notation}\label{suppressRnotation}
We will often suppress our fixed ring $R$ from the notation and write $C^*(X)$ (or $C^*_c(X)$) for the (compactly supported) cochains of $X$ with coefficients in the constant sheaf $R$, with corresponding notation for the cohomology of these chain complexes. We similarly write $C_*(X)$ (or $C^\lf_*(X)$) for the (locally finite) chains of $X$ with coefficients in the constant cosheaf $R$, extended to notation for the homology in the obvious way. 
\end{notation}

\subsubsection{Relative sheaf cohomology and cosheaf homology}\label{relativecohsect}
Let $L$ be a subcomplex of $X$. For a combinatorial sheaf $\F$ or cosheaf $\G$ on $X$ we have the restrictions $\F|_{L}$ and $\G|_L$ to $L$, with the associated (co)chain complexes $C^{\bullet}(L,\F|_L)$ and $C_\bullet(L,\G|_L)$ and (co)homology $H^*(L,\F|_L)$ and $H_*(L,\G|_L)$. The inclusion map $i\colon L\hookrightarrow X$ induces maps
\begin{align*}
	i^\bullet\colon&C^{\bullet}(X,\F)\to C^{\bullet}(L,\F|_L)\\
	i_\bullet\colon&C_\bullet(L,\G|_L) \to C_\bullet(X,\G).
\end{align*}
Using these maps, we define the relative (co)homology of the pair $(X,L)$ to be
\begin{align*}
	H^*(X,L;\F)&:= H^*(\textnormal{Ker }i^\bullet)\\
	H_*(X,L;\G)&:= H_*(\textnormal{Coker }i_\bullet).
\end{align*}
The compactly supported and locally finite versions are defined analogously. We will extend the convention from Notation~\ref{suppressRnotation} and often suppress $R$ from the notation if we are taking coefficients in the constant (co)sheaf.

\subsubsection{Orientation independence of (co)homology}\label{orindepcoh}
While the differentials for sheaf cohomology (Equation \eqref{sheafboundary}) and cosheaf homology (Equation \eqref{coshcohdiff}) depend on a choice of orientation of $X$ the (co)homology does not. A change of orientation gives for each $k$-simplex a permutation of its vertices, and conversely a change of orientation for a locally finite complex is specified by a (possibly infinite) collection of permutations. 

\begin{lemma}\label{homorderinv}
	Let $X$ be a locally finite simplicial complex and let $\mathcal{F}$ be sheaf on $X$. Pick two orientations of $X$ and denote the corresponding simplicial chain complexes by $C^\bullet(X;\mathcal{F})$ and $\tilde{C}^\bullet(X;\mathcal{F})$. Then the map
	\begin{align*}
		C^k(X;\mathcal{F})&\rar \tilde{C}^k(X;\mathcal{F})\\
		\sum_{\smallsimplex}\shcochain_\smallsimplex &\mapsto \sum_{\smallsimplex} \sg{\smallsimplex} \shcochain_{\smallsimplex}
	\end{align*}
	where $\sg{\smallsimplex}$ is the sign of the permutation corresponding to the change of orientation on $\smallsimplex$, is an isomorphism of chain complexes.
\end{lemma}
\begin{proof}
	Left as an exercise.
\end{proof}

The obvious analogue of this map induces an isomorphism in homology.

\subsection{The local homology sheaf and the local cohomology cosheaf}\label{lochomshsect}
We now explain how to view local homology as a sheaf and local cohomology as a cosheaf in \textsection\ref{lochomsheafsect}, and provide tools for computing them in \textsection\ref{lochomcompsection}.
\subsubsection{Local homology and cohomology}\label{lochomsheafsect}
Recall that the \emph{open star} $\ostar(\smallsimplex)$ of a simplex $\smallsimplex$ is the set of (open) simplices $\smallsimplexone$ such that $\smallsimplex \leq \smallsimplexone$. The complement $X- \ostar(\smallsimplex)$ is then the simplicial subcomplex of $X$ spanned by the simplices of $X$ that do not contain $\smallsimplex$. The local homology of $X$ at a simplex $\smallsimplex$ is the graded $R$-bimodule
\[ h_*(\smallsimplex)=H_*(X,X-\ostar(\smallsimplex)). \] 
If $\smallsimplex < \smallsimplexone$ then $\ostar(\smallsimplexone) \subset \ostar(\smallsimplex) $ and $X- \ostar(\smallsimplex) \subset X-\ostar(\smallsimplexone)$. Inclusion of pairs induces a morphism 
\[h_*(\smallsimplex < \smallsimplexone) \co H_*(X, X- \ostar(\smallsimplex)) \to H_*(X,X-\ostar(\smallsimplexone)),\] 
making $h_*$ into a combinatorial sheaf with values in graded $R$-bimodules called the \emph{local homology sheaf} of $X$. The \emph{local cohomology cosheaf} is defined by replacing homology with cohomology, so that 
$$
h^*(\smallsimplex)=H^*(X, X- \ostar(\smallsimplex)). 
$$
When $\smallsimplex < \smallsimplexone$ there is an induced morphism $h^*(\smallsimplexone>\smallsimplex)\co h^*(\smallsimplexone) \to h^*(\smallsimplex)$ given by the same inclusion of pairs as above.

For a fixed $n$ there are associated sub (co)sheaves $h_n$ and $h^n$ given by taking (co)homology in a given degree. In this paper, we will mostly be working with spaces where there exists $n$ such that $h_*=h_n$ and $h^*=h^n$, and as a result may blur this distinction. 

When $\mathcal{G}$ is a cosheaf the \emph{local homology sheaf  $(h_\mathcal{G})_*$ of $X$ with coefficients in $\mathcal{G}$} is the sheaf of right $R$-modules obtained analogously by taking relative homology with coefficients in $\mathcal{G}$. Similarly when $\mathcal{F}$ is a sheaf we have the \emph{local cohomology cosheaf  $(h^\mathcal{F})^*$ of $X$ with coefficients in $\mathcal{F}$}.

\subsubsection{Computing $h_n(\smallsimplex)$ and $h^n(\smallsimplex)$}\label{lochomcompsection}
For any simplex $\smallsimplex  \in X$, the complex $C_\bullet(X-\ostar(\smallsimplex))$ is the subcomplex of $C_\bullet(X)$ spanned by simplices that do not contain $\smallsimplex$, so the quotient complex $C_\bullet(X,X-\ostar \smallsimplex)=C_\bullet(X)/ C_\bullet(X-\ostar(\smallsimplex))$ is freely generated by simplices that contain $\smallsimplex$. Cf. Notation~\ref{suppressRnotation} we have suppressed our fixed ring $R$ from the notation here. We introduce the following:
\begin{notation}\label{undernotation}
	For a simplex  $\smallsimplex \in X$ we denote by $\smallsimplex\under \bigsimplex \in C_k(X,X-\ostar \smallsimplex)$ the element associated to a $k$-simplex $\bigsimplex\in X_k$. Note that $\smallsimplex\under \bigsimplex=0$ whenever $\smallsimplex \nleq \bigsimplex$.
\end{notation}

We record the following observation:

\begin{lemma}\label{alwayszerolochom}
	Let $\smallsimplex \in X_l$. Then $h_k(\smallsimplex)=0$ for $k<l$. Additionally, if $\smallsimplex$ is not maximal then $h_l(\smallsimplex)=0$.
\end{lemma}
When $X$ is $n$-dimensional, we have that 
\begin{equation}\label{hndescript}
	h_n(\smallsimplex) = \ker \left(C_n(X,X-\ostar \smallsimplex)\xrightarrow{d} C_{n-1}(X,X-\ostar \smallsimplex)\right).	
\end{equation}
Remembering that $X$ is assumed to be locally finite, every element $\localchain \in h_n(\smallsimplex)$ can be written as a sum $\localchain= \sum r_i \smallsimplex\under\bigsimplex_i$, where $r_i \in R$, and $\sum r_i \partial\bigsimplex_i$ is an element of $C_{n-1}(X-\ostar(\smallsimplex))$.	
If $\smallsimplex<\smallsimplexone$ then the sheaf maps have the simple form $h_n(\smallsimplex < \smallsimplexone)(\sum r_i\smallsimplex\under \bigsimplex_i)=\sum r_i \smallsimplexone \under\bigsimplex_i$, where the only difference is that the homology class is relative to $X-\ostar(\smallsimplexone)$ in the image rather than $X-\ostar(\smallsimplex)$. In particular, $h_n(\smallsimplex< \smallsimplexone)$ takes $\smallsimplex\under \bigsimplex_i$ to zero if $\bigsimplex_i$ does not contain $\smallsimplexone$. 

Recall that a ring $R$ is called \emph{hereditary} if every submodule of a projective module is projective. Any principal ideal domain is hereditary. We have:
\begin{lemma}\label{hnprojectivelemma}
	Let $R$ be a hereditary ring. If $X$ is a complex of dimension $n$, then $h_n(\smallsimplex)$ is projective (as a left $R$-module) for every $\smallsimplex \in X$.
\end{lemma}
\begin{proof}
 Note that $h_n(\smallsimplex)$ is a submodule of the free $R$-module $C_n(X,X-\ostar \smallsimplex)$.
\end{proof}

We will use $\bigsimplex^*$ to denote the basis element of $C^\bullet(X)$ corresponding to a simplex $\bigsimplex$. The complex $C^\bullet(X,X-\ostar \smallsimplex)$ is the subcomplex of $C^\bullet(X)$ spanned by the $\bigsimplex^*$ such that $\bigsimplex$ contains $\smallsimplex$. When $X$ is $n$-dimensional the local cohomology in degree $n$ has the simple description 
$$
h^n(\smallsimplex) =\coker \left( C^{n-1}(X,X-\ostar \smallsimplex) \xrightarrow{\delta} C^n(X,X-\ostar \smallsimplex) \right). 
$$
As above, this allows one to write the cosheaf maps from $h^n(\smallsimplex)$ to $h^n(\smallsimplexone)$  as $h^n( \smallsimplex>\smallsimplexone)[\sum r_i \smallsimplex\under \bigsimplex_i^*]=[\sum r_i \smallsimplexone \under \bigsimplex_i^*]$, extending Notation~\ref{undernotation} in the obvious way. 

\begin{notation}[Abuse of notation]
	In the rest of this paper it will often be convenient to work in terms of the \emph{(co)chains} $\smallsimplex\under \bigsimplex$ (or $\smallsimplex\under\bigsimplex^*$) for the local (co)homology, rather than the classes. As a result, maps will often be implicitly defined at the level of $C_*(X,X-\ostar\smallsimplex)$ (or $C^*(X,X-\ostar\smallsimplex)$) rather than $h_*(\smallsimplex)$ (or $h^*(\smallsimplex)$).
\end{notation}

\subsection{Sheaf cohomology and cosheaf homology, and simplicial homology}
The reader may have noticed that the chain complexes introduced in \textsection\ref{homcohomsimplsh} are similar to those for simplicial (co)homology. We will make use of this similarity in our proofs. To be able to do so, we capture the similarity formally using the notions of \emph{$\mathcal{G}$-transport} and \emph{$\mathcal{F}$-dual}:

\begin{definition}[$\mathcal{G}$-transport and $\mathcal{F}$-dual] \label{d:gtransport}\label{d:fdual}
	Consider the category $\mathcal{C}(X)$ with objects given by free $R$-modules $A$ on a set $I_A$ indexed by elements of the form $\smallsimplex\under \bigsimplex$ with $\smallsimplex<\bigsimplex$ simplices in $X$ (we allow $\emptyset\under \bigsimplex$ and set $\emptyset\under \bigsimplex=\bigsimplex$). If $B \in \mathcal{C}(X)$ has basis $I_B$, a morphism from $A$ to $B$ in $\mathcal{C}(X)$ is a homomorphism $f\colon A \to B$ such that $f$ has a matrix representation with respect to the bases $I_A$ and $I_B$. The matrix entries $f^{\smallsimplex\under \bigsimplex}_{\smallsimplexone\under\bigsimplexone}$ for $\smallsimplex\under \bigsimplex \in I_A$ and $\smallsimplexone\under\bigsimplexone \in I_B$ are required to be in the image of $\Z \rar R$ and $0$ whenever $\bigsimplex\ngeq\bigsimplexone$. 
	
	Given a cosheaf $\mathcal{G}$ we have a a covariant additive functor called \emph{$\mathcal{G}$-transport} that sends an object $A\in \mathcal{C}(X)$ to $A_\mathcal{G}=\bigoplus_{\smallsimplex\under \bigsimplex \in I_A} \mathcal{G}(\bigsimplex)$, and we will write $\smallsimplex\under \psi_\bigsimplex$ for elements of $\mathcal{G}(\bigsimplex)$ in the $\smallsimplex\under \bigsimplex$-summand. A morphism $f\colon A \rar B$ in $\mathcal{C}(X)$ is sent to the map
	\begin{align*}
		f_\mathcal{G}\colon& \bigoplus_{\smallsimplex\under \bigsimplex \in I_A} \mathcal{G}(\bigsimplex)\rar  \bigoplus_{\smallsimplexone\under\bigsimplexone \in I_B} \mathcal{G}(\bigsimplexone)\\
		&\smallsimplex\under \psi_\bigsimplex \mapsto \sum_{\smallsimplexone\under\bigsimplexone} f^{\smallsimplex\under \bigsimplex}_{\smallsimplexone\under\bigsimplexone}\smallsimplexone \under \mathcal{G}(\bigsimplex>\bigsimplexone)\psi_\bigsimplex.
	\end{align*}

    We similarly have a version of $\mathcal{G}$-transport defined on the category $\mathcal{C}^\sqcap(X)$ whose objects are products $A=\prod_{I_A}R$, where $I_A$ is as before, and whose morphisms are module maps $f\colon A\to B $ such that for all $\smallsimplex \under \bigsimplex \in I_A$ we can decompose $f(\smallsimplex\under\bigsimplex)= \sum_{\smallsimplexone \under \bigsimplexone}f^{\smallsimplex\under\bigsimplex}_{\smallsimplexone\under\bigsimplexone} \smallsimplexone\under\bigsimplexone$ as a sum with finitely many nonzero terms with integer coefficients, and such that for any $\smallsimplexone\under \bigsimplexone$ the coefficients $f^{\smallsimplex\under\bigsimplex}_{\smallsimplexone\under\bigsimplexone} $ are nonzero for finitely many $\smallsimplex\under\bigsimplex$, and zero whenever $\bigsimplex\ngeq \bigsimplexone$.
    
	Given a sheaf $\mathcal{F}$ the contravariant additive functor called \emph{$\mathcal{F}$-dual} sends $(A,I_A)\in \mathcal{C}(X)$ to $A^\mathcal{F}= \prod_{\smallsimplex\under \bigsimplex \in I_A} \hom(\langle \smallsimplex\under \bigsimplex \rangle, \mathcal{F}(\bigsimplex))$, we will write $\smallsimplex\under \phi^\bigsimplex$ for elements of $\hom(\langle \smallsimplex\under \bigsimplex \rangle, \mathcal{F}(\bigsimplex))$. A morphism $f\colon A \rar B $ in $\mathcal{C}(X)$ is sent to the map
	\begin{align*}
		f^\mathcal{F}\colon& \prod_{\smallsimplexone\under \bigsimplexone \in I_B} \mathcal{F}(\bigsimplexone)\rar  \prod_{\smallsimplex\under \bigsimplex \in I_B} \mathcal{F}(\bigsimplex)\\
		&\smallsimplexone \under \phi^\bigsimplexone \mapsto \sum f^{\smallsimplex\under \bigsimplex}_{\smallsimplexone\under\bigsimplexone}\smallsimplex\under \mathcal{F}(\bigsimplexone< \bigsimplex)\phi^\bigsimplexone.
	\end{align*}
	We similarly have a category $\mathcal{C}^*(X)$ where objects are spanned by elements of the form $\smallsimplex\under \bigsimplex^*$ and matrix entries $f^{\smallsimplex\under \bigsimplex^*}_{\smallsimplexone\under\bigsimplexone^*}$ vanish whenever $\bigsimplex\nleq\bigsimplexone$ and an associated covariant additive functor of \emph{$\F$-transport} that we will denote by $(-)_\mathcal{F}$. This extends to the category $\mathcal{C}_\sqcap^*(X)$, defined analogously to  $\mathcal{C}^\sqcap(X)$, where the objects are products indexed by sets of elements of the form $\smallsimplex\under \bigsimplex^*$. 
\end{definition}

One checks that if $f,g$ are composable morphisms in $\mathcal{C}(X)$ we have $(fg)_\mathcal{G}=f_\mathcal{G} \circ g_\mathcal{G}$ and $(fg)^\mathcal{F}=g^\mathcal{F} \circ f^\mathcal{F}$. Observe that as $C_k(X)$ is spanned by $\{\smallsimplex\under \smallsimplex\}_{\smallsimplex\in X_k}$:
\begin{align*}
	C_\bullet(X;\G)&=C_\bullet(X)_\G\\
	C^\bullet(X;\F)&=C_\bullet(X)^\F=C^\bullet(X)_\F,
\end{align*}
and similarly for relative (co)homology. When (as will be the case for Cohen--Macaulay spaces) the local homology modules of $X$ are free $R$-modules, we further have
\begin{align*}
	C_\bullet(X;(h_\G)_*)&=C_\bullet(X; h_*)_\G\\
	C^\bullet(X;(h^\F)^*)&=C_\bullet(X;h^*)^\F,
\end{align*}
for the (co)homology of the local (co)homology sheaf with coefficients.

\subsection{The cap product}\label{capproductsection} 
In this section we define cap products for local (co)homology. There will be several variants, half of which use a pairing between the local homology and local cohomology introduced in \textsection\ref{pairingsect}. After giving a chain level formula for the cap products in \textsection\ref{simplicialcapsection}, we show that they descend to (co)homology by showing that they define a chain map. The chain level definition of the cap products depends on an orientation of the simplicial complex: we show in \textsection\ref{orientationindepcapsection} that in (co)homology this orientation dependence disappears. After this, we define the cap product for relative (co)homology in \textsection\ref{relativecapssection}.

\subsubsection{Pairing between local homology and cohomology}\label{pairingsect}
For a fixed degree $n$ and any simplex $\smallsimplex$ there is the usual pairing (recall that $\otimes=\otimes_R$)
\begin{equation}\label{pairingone}
	\langle\cdot,\cdot\rangle\colon h^n(\smallsimplex) \otimes h_n(\smallsimplex) \to R
\end{equation}	
evaluating a relative cohomology class on a relative homology class. Note that this pairing defines a bimodule map. 

\subsubsection{Simplicial cap product}\label{simplicialcapsection}
There are two versions of the cap product. The first uses the pairing from Equation~\eqref{pairingone} to give maps
\begin{align*} 
	H^\textnormal{lf}_k(X;h^n)&\otimes H^l(X;h_n)\xrightarrow{\cap} H^\textnormal{lf}_{k-l}(X;R)\\
	H_k^{\textnormal{lf}}(X;h^n) &\otimes H^l_c(X;h_n) \xrightarrow{\cap} H_{k-l}(X;R),
\end{align*}
where $l \leq k$. These maps are both defined at the chain level by the formula (recall Notation~\ref{simplexnotation} and the assumption that $X$ is  locally finite and oriented)
\begin{equation}\label{cupversion1}
	\shchain \cap \shcochain=\sum_{\smallsimplex \in X_k} \langle h^*( \smallsimplex>\smallsimplex_{\geq k- l})\shchain_\smallsimplex, \shcochain_{\smallsimplex_{\geq k- l}} \rangle \smallsimplex_{\leq k- l},
\end{equation}
where $\shchain$ is a locally finite chain and $\shcochain$ a (compactly supported) cochain, with values $\chain_\smallsimplex$ and $\shcochain_\smallsimplex$ at $\smallsimplex$. On generators $\smallsimplex\under \bigsimplex^*\in h^*(\smallsimplex)$ where $\smallsimplex$ is a $k$-simplex and $\smallsimplexone \under\bigsimplexone\in h_*(\smallsimplexone)$ where $\smallsimplexone$ is an $l$-simplex this pairing gives 
\begin{equation}\label{cap1ongens}
	\smallsimplex\under \bigsimplex^*\cap \smallsimplexone \under\bigsimplexone =\begin{cases}
		\langle \bigsimplex^*,\bigsimplexone\rangle \smallsimplex_{\leq k-l} & \mbox{ if }\smallsimplex_{\geq k-l}= \smallsimplexone\\ 0 & \mbox{ otherwise.}
	\end{cases}
\end{equation}

The second version does not use the pairing $\langle \,,\,\rangle$ and gives for $l \leq k$ maps
\begin{align}\begin{split}
		H^\lf_k(X;h^n)&\otimes H^l(X;R)\xrightarrow{\cap} H^\lf_{k-l}(X;h^n)\\
		H_k^{\textnormal{lf}}(X;h^n) &\otimes H^l_c(X;R) \xrightarrow{\cap} H_{k-l}(X;h^n),\label{capnopairing}	
	\end{split}
\end{align}
defined by
\begin{equation}\label{cupversion2}
	\shchain \cap \intcochain= \sum_{\smallsimplex \in X_k}\intcochain(\smallsimplex_{\geq k-l})h^n( \smallsimplex> \smallsimplex_{\leq k-l})\shchain_\smallsimplex,	
\end{equation}
where $\shchain$ is a locally finite chain with coefficients in $h^*$ and $\intcochain$ a (compactly supported) cochain with $R$-coefficients. On elements $\smallsimplex\under \bigsimplex^*\in C^\bullet(X,X-\ostar \smallsimplex)$ used to define elements of $h^*(\smallsimplex)$ where $\smallsimplex\in X_k$ and $\smallsimplexone^*$ where $\smallsimplexone \in X_l$, this gives
\begin{equation}\label{cap2ongens}
	\smallsimplex \under \bigsimplex^* \cap \smallsimplexone^*=\begin{cases}
		\smallsimplex_{\leq k- l}\under\bigsimplex^* &\mbox{ if }\smallsimplex_{\geq k-l}=\smallsimplexone\\ 0& \mbox{ otherwise.}
	\end{cases} 
\end{equation}	

These cap products are chain maps in the sense that:
\begin{proposition}[Leibniz Rule]\label{capischainmap}
	Let $X$ be a locally finite oriented simplicial complex. The chain-level cap product
	$$
	C^\lf_k(X;h^*) \otimes C^l(X;h_*) \xrightarrow{\cap} C^\lf_{k-l}(X;R)	
	$$
	defined in Equation~\eqref{cupversion1} is a chain map with respect to the differential of the tensor product complex given by
	$$
	\partial_\textnormal{tot}(\shchain\otimes \cochain)= d\shchain \otimes\cochain + (-1)^{k-l}\shchain\otimes \partial \cochain.
	$$
	The other cap products in this section are also chain maps with a similarly defined differential on the respective tensor product complexes.
\end{proposition}
\begin{proof}
	Elements of $C^\lf_k(X;h^*)$ and $ C^l(X;h_*)$ can be represented by linear combinations of elements of the form $\smallsimplex\under \bigsimplex^*\in C^\bullet(X,X-\ostar \smallsimplex)$ and $\smallsimplexone\under \bigsimplexone\in  C_\bullet(X,X-\ostar \smallsimplexone)$, respectively. We will show that
	$$
	d(\smallsimplex\under \bigsimplex^*\cap \smallsimplexone \under\bigsimplexone)=d(\smallsimplex\under \bigsimplex^*)\cap \smallsimplexone \under \bigsimplexone+(-1)^{k-l}\smallsimplex\under \bigsimplex^*\cap \partial(\smallsimplexone \under\bigsimplexone).
	$$
	The left hand side gives (suppressing the (co)restriction maps from the notation)
	\begin{equation}\label{dofcap}
		d(\smallsimplex\under \bigsimplex^*\cap \smallsimplexone \under\bigsimplexone)= \delta_{\smallsimplex_{\geq k-l},\smallsimplexone} \langle \bigsimplex^*,\bigsimplexone\rangle \left(\sum_{i=0}^{k-l-1}(\smallsimplex_{\leq k- l})_{\langle i \rangle}+(-1)^{k-l}\smallsimplex_{\leq k-l-1}\right),
	\end{equation}
	where $\delta_{\smallsimplex_{\geq k-l},\smallsimplexone}$ is the Kronecker delta. Meanwhile
	\begin{align}\begin{split}\label{capwithd}
	d(\smallsimplex\under \bigsimplex^*)\cap \smallsimplexone \under \bigsimplexone=&\sum_{i=0}^{k-l-1}(-1)^i \delta_{\smallsimplex_{\geq k-l}, \smallsimplexone}\langle \bigsimplex^*,\bigsimplexone\rangle (\smallsimplex_{\langle i \rangle})_{\leq k-l-1} \\
		&+\sum_{i=k-l}^k(-1)^i \delta_{(\smallsimplex_{\langle i \rangle})_{\geq k-l-1}, \smallsimplexone}\langle \bigsimplex^*,\bigsimplexone\rangle \smallsimplex_{\leq k-l-1},
	\end{split}\end{align}
	and one sees that the first term of this cancels with the first term on the right hand side of Equation~\eqref{dofcap}. The final term is $(-1)^{k-l}$ times
	\begin{align*}
		\smallsimplex\under \bigsimplex^*\cap \partial\smallsimplexone \under\bigsimplexone=& \delta_{\smallsimplex_{\geq k-l}, \smallsimplexone} \langle \bigsimplex^*,\bigsimplexone\rangle\smallsimplex_{\leq k-l-1}\\
		&+\sum_{i=k-l}^k(-1)^{i-k+l+1} \delta_{(\smallsimplex_{\langle i \rangle})_{\geq k-l-1}, \smallsimplexone} \langle \bigsimplex^*,\bigsimplexone\rangle \smallsimplex_{\leq k-l-1}.
	\end{align*}
	We see that the first term of this cancels with the second term from Equation~\eqref{dofcap}, while the second term cancels with the second in Equation~\eqref{capwithd}.
	
	The proofs for the other cap products are analogous.
\end{proof}

In particular we have found that for $\shchain\in C^\lf_k(X;h^*)$ and $\cochain\in C^l(X;h_*)$
$$
d(\shchain\cap \cochain)=  d\shchain \cap \cochain + (-1)^{k-l}\shchain\cap \partial \cochain,
$$
and similar for the other versions of the cap product. From this it follows that:
\begin{corollary}
	The cap products defined in this section descend to (co)homology.
\end{corollary}

\subsubsection{Orientation independence of the cap product}\label{orientationindepcapsection}
In defining the cap product we used the orientation of the simplicial complex $X$. Just like the (co)homology itself, the cap products at the level of (co)homology do not depend on this choice:

\begin{theorem}\label{orderderindependencetheorem}
	For a locally finite oriented simplicial complex the cap products from the previous section do not depend on the chosen orientation at the level of (co)homology.
\end{theorem}

The proof of this relies on the following proposition, the proof of which will be given in Appendix~\ref{orindepapp}.

\begin{proposition}\label{orindepprop}
		Let $X$ be an oriented simplicial complex and $[u,w]$ be an oriented 1-simplex. Write $\tilde{C}_\bullet(X)$ for the chain complex associated to the orientation of $X$ in which $u$ and $w$ are swapped and $\tilde{\cap}$ for the associated cap product. Then the diagram involving the cap from Equation~\eqref{cupversion2}
	\begin{center}
		\begin{tikzcd}
			C^\lf_k(X;h^*)\otimes C^l(X) \arrow[r,"\cap"] \arrow[d,"\cong"] & C^\lf_{k-l}(X;h^*)\arrow[d, "\cong"]\\
			\tilde{C}^\lf_k(X;h^*)\otimes\tilde{C}^l(X) \arrow[r,"\tilde{\cap}"]& \tilde{C}^\lf_{k-l}(X;h^*)
		\end{tikzcd}
	\end{center}
	(where the vertical maps are the isomorphisms from Lemma~\ref{homorderinv}) commutes up to chain homotopy. The corresponding diagram for the cap product from Equation~\eqref{cupversion1} also commutes. The analogous statements for the cap product involving compactly supported cohomology hold.
\end{proposition}

\begin{remark}
	In Appendix~\ref{orindepapp} we also prove the corresponding statements for the cap product for (co)homology with $R$-coefficients.
\end{remark}

\begin{proof}[Proof of Theorem~\ref{orderderindependencetheorem}]
	The cap product on a generator $\smallsimplex\otimes\smallsimplexone^*$ only depends on the orientation of the simplices of $\smallsimplex$ and $\smallsimplexone$, which span a finite subcomplex. A change of orientation of $X$ restricted to this finite subcomplex is given by a product of transpositions. By Proposition~\ref{orindepprop} the cap products in (co)homology are invariant under changing the orientation by a transposition, so we are done.
\end{proof}

\begin{remark}
	In the literature (e.g. \cite[Page 239]{Hatcher2002}) one finds the cap product
	$$
	\smallsimplex\cap \smallsimplexone^*= \smallsimplexone^*(\smallsimplex_{\leq l})\smallsimplex_{\geq l},
	$$
	given by evaluating the cochain on the front face, rather than the back face of $\smallsimplex$. This is not a chain map for a reasonable choice of signs in the chain complexes involved. It is related to our cap product by reversing the order of the vertices in each of the simplices of $X$. The corresponding permutation for  an $l$-simplex has sign $(-1)^{\frac{1}{2}(l-2)(l-1)}$, which depends on $l$ mod 4.
\end{remark}

\subsubsection{Relative caps}\label{relativecapssection} We will also need versions of the cap product relative to subcomplexes. Because the cap product makes reference to the front and back faces of simplices with respect to the orientation of the simplicial complex $X$, some care needs to be taken in how one orients the subcomplexes:

\begin{definition}\label{vcandbeforedef}
	Let $L\subset X$ be a subcomplex of an oriented simplicial complex $X$. The \emph{vertex-complement $L^{\textnormal{vc}}$ of $L$} is the full subcomplex spanned by the vertices of $X$ not in $L$. We will say that \emph{$L\vc$ is before $L$} if $X$ is oriented so that in each simplex the vertices of $L\vc$ come before those of $L$.
\end{definition}  

We record some obvious features of orienting $L\vc$ before $L$ to facilitate cap product computations:
\begin{lemma}\label{basicorderingconsequences}
	Let $L\subset X$ and assume that $L\vc$ is before $L$. Let $\smallsimplex\in X_k$. Then, for $0\leq j \leq k$,
	\begin{enumerate}[label=(\arabic*), ref=\arabic*]
		\item if $\smallsimplex_{\leq j} \in L$, then $\smallsimplex \in L$;
		\item if $\smallsimplex_{\leq j}\notin L\vc$, then $\smallsimplex_{\geq j}\in L$;
		\item if $\smallsimplex_{\geq j} \in L\vc$, then $\smallsimplex \in L\vc$;
		\item if $\smallsimplex_{\geq j} \notin L$, then $\smallsimplex_{\leq j}\in L\vc$.
	\end{enumerate}
\end{lemma}

With this orientation, we have the relative cap products:
\begin{proposition}\label{usefulrelativecups}
	Let $L\subset X$ be a subcomplex, and let the orientation of $X$ be so that $L\vc$ is before $L$. Then Equations \eqref{cupversion1} and \eqref{cupversion2} induce chain level relative caps
	\begin{align}
		C^\lf_k(X; h^n)\otimes C^{l}(X,L; h_n)& \xrightarrow{\cap} C^\lf_{k-l}(L\vc; R),\label{relativecup1}\\
		C^\lf_k(X; h^n)\otimes C^{l}(L; h_n) &\xrightarrow{\cap} C^\lf_{k-l}(X,L\vc; R)\label{relativecup2}\\
		C^\lf_k(X; h^n)\otimes C^{l}(X,L; R)& \xrightarrow{\cap} C^\lf_{k-l}(L\vc; h^n)\label{relativecap3}\\
		C^\lf_k(X; h^n)\otimes C^{l}(L; R) &\xrightarrow{\cap} C^\lf_{k-l}(X,L\vc; h^n),\label{relativecap4}
	\end{align} 
	which descend to (co)homology. These restrict on compactly supported cochains to maps that take values in finite chains.
\end{proposition}
\begin{proof}
	We only discuss the first two maps, the second two are similar. 
	We start by giving the relative cap from Equation~\eqref{relativecup1}. We want to show that, given a chain $\shchain\in C_k(X;h^*)$, capping with $\shchain$ takes $C^l(X,L;h_n)$ to $C_{k-l}(L\vc;R)$. For any cochain $\cochain \in C^{l}(X,L;h_n)$ and $k-l$-simplex $\smallsimplexone \notin L\vc_{k-l}$, let $(\shchain\cap \cochain)_\smallsimplexone$ be the $\smallsimplexone$-component of the cap product . This component has a contribution from each $k$-simplex $\smallsimplex$ for which $\smallsimplexone=\smallsimplex_{\leq k-l}$. As $L\vc$ is before $L$ and $\smallsimplexone=\smallsimplex_{\leq k-l} \not \in L\vc$, this implies that $\smallsimplex_{\geq k- l} \in L_l$ (Lemma~\ref{basicorderingconsequences}(2)). Hence $\cochain(\smallsimplex_{\geq k- l})=0$ and thus $(\shchain\cap \cochain)_\smallsimplexone=0$. As this holds for all $\tau \not \in L\vc$, this shows that $\shchain\cap \cochain\in C_{k-l}(L\vc;R)$.

	For Equation~\eqref{relativecup2}, recall that both
	\begin{align}\begin{split}\label{usefulses}
			0\rightarrow C^l(X,L;h_n|_L)\rightarrow C^l(X; h_n|_L)\rightarrow C^l(L; h_n|_L)\rightarrow 0	& \mbox{ and}\\
			0\rightarrow C_{k-l}(L\vc;R)\rightarrow C_{k-l}(X;R)\rightarrow C_{k-l}(X,L\vc;R) \rightarrow 0&
		\end{split}
	\end{align}
	are short exact sequences of chain complexes. The argument for Equation~\eqref{relativecup1} shows that capping takes the kernel of the first sequence to the kernel of the second, so the cap product descends to the quotients as
	$$
	C_k(X; h^n)\otimes C^{l}(L;h_n) \xrightarrow{\cap} C_{k-l}(X,L\vc; R).
	$$
	Using Proposition~\ref{capischainmap} we see that this descends to (co)homology.
\end{proof}

\subsubsection{The cap product with coefficients}\label{s:capwithcoefs}
For a cosheaf $\G$ using $\mathcal{G}$-transport (Definition~\ref{d:gtransport}) gives caps
\begin{align*}
	C^\lf_k(X; h^{n})\otimes C^l(L;h^\mathcal{G}_n|_L)& \xrightarrow{\cap_\mathcal{G}} C^{\lf}_{k-l}(X,L\vc;\mathcal{G})	\\
	C^\lf_k(X; h^{n})\otimes C^l(X,L\vc;h^\mathcal{G}_n) &\xrightarrow{\cap_\mathcal{G}} C^\lf_{k-l}(L;\mathcal{G}).
\end{align*}
given by applying $\mathcal{G}$-transport to Equation~\eqref{cap1ongens} for a fixed $\smallsimplex\under \bigsimplex^*$:
$$
\smallsimplex\under \bigsimplex^*\cap_\mathcal{G}  \smallsimplexone\under\psi_\bigsimplexone  = \begin{cases}
	\mathcal{G}(\bigsimplex > \smallsimplex_{\leq k-l})\psi_\bigsimplex &\mbox{ for }\bigsimplex=\bigsimplexone \mbox{ and }\smallsimplexone=\smallsimplex_{\geq k-l}\\ 0 & \mbox{ otherwise.}
\end{cases}
$$
We note that $\mathcal{G}$-transport applies here as for a given $\smallsimplexone$ there are only finitely many $\smallsimplex$ such that $\smallsimplex_{\geq k-l}=\smallsimplexone$ as we assume our complex to be locally finite. Like the cap from Equation~\ref{cap1ongens}, the cap product $\cap_\mathcal{G}$ restricts to a cap product involving compactly supported $\mathcal{G}$-valued cochains. 

Similarly, we have for a sheaf $\mathcal{F}$  a version of the cap product from Equation~\eqref{cupversion2}
$$
C^\lf_k(X;h^*)\otimes C^l(X;\mathcal{F})\xrightarrow{\cap_\mathcal{F}} C^\lf_{k-l}(X;h_\mathcal{F}^*),
$$
with obvious relative and compactly supported versions. The formula for this cap product is given by the $\mathcal{F}$-transport of Equation~\eqref{cap2ongens}:
$$
\smallsimplex\under\bigsimplex^* \cap \smallsimplexone \under \phi^\smallsimplexone= \begin{cases} \smallsimplex_{\leq k-l}\under\mathcal{F}(\smallsimplex_{\geq k-l}<\bigsimplex)\phi^{\smallsimplex_{\geq k-l}}  &\mbox{ for }\smallsimplex_{\geq k-l}=\smallsimplexone\\ 0&\mbox{ otherwise.}
\end{cases}
$$

\section{Overview of Cohen--Macaulay Complexes}\label{CMsection}
We will discuss the basics of Cohen--Macaulay complexes in Section~\ref{CMcondsect}, and then proceed to give a characterisation of the zeroth homology of the local cohomology cosheaf of a CM complex in terms of homomorphisms on sections of the local homology sheaf in Section~\ref{cmpropssection}. 

\subsection{Cohen--Macaulay complexes}\label{CMcondsect}
We discuss the Cohen--Macaulay condition and some basic consequences in \textsection\ref{cmcondsect}, give examples of CM complexes in \textsection\ref{cmexsect}, and discuss a characterisation of the CM condition in terms of links in \textsection\ref{linkcharsect}.

\subsubsection{The (homological) Cohen--Macaulay condition}\label{cmcondsect}
CM complexes are defined through a condition on the local homology at simplices:

\begin{definition}[Cohen--Macaulay complex]\label{d: CM}
A simplicial complex $X$ is \emph{Cohen--Macaulay of dimension $n$} over a ring $R$ if the following two conditions hold:

\begin{enumerate}[label=(\arabic*), ref=\arabic*]
	\item For each nontrivial simplex $\sigma$ in $X$, the local homology $h_*(\smallsimplex)$ over $R$ is concentrated in degree $n$.
	\item The reduced homology $\overline{H}_*(X)$ is concentrated in degree $n$.
\end{enumerate}
We say that $X$ is \emph{locally Cohen--Macaulay} (or \emph{locally CM} for short), if $X$ satisfies condition (1). If $L$ is a subcomplex of $X$, we say that $X$ is \emph{locally CM of dimension $n$ at $L$} if $h_*|_L$ is concentrated in degree $n$.

Similarly, we say that a simplicial complex $X$ is \emph{locally CM of dimension $n$ over the cosheaf $\mathcal{G}$ at a subcomplex $L$} if $h_*^\mathcal{G}(\smallsimplex)$ is concentrated in degree $n$ for all $\smallsimplex\in L$. Analogously, $X$ is \emph{locally CM of dimension $n$ over the sheaf $\mathcal{F}$ at a subcomplex $L$} if $h^*_\mathcal{F}(\smallsimplex)$ is concentrated in degree $n$ for all $\smallsimplex\in L$.
\end{definition}

\begin{remark}
	In many papers these two conditions are combined into one by noting that the open star $\ostar(\emptyset)$ of the empty simplex $\emptyset$ is the whole space, so that $\mathcal{H}(\emptyset)=\overline{H}_*(X)$. We separate the two statements as we consider (1) to be a \emph{local} condition on $X$, whereas (2) is a global condition.
\end{remark}

A finite dimensional simplicial complex is called \emph{pure} if all maximal simplices are of the same dimension. 
\begin{lemma}\label{pureremark}
	Any simplicial complex $X$ that is locally CM of dimension $n$ is pure of dimension $n$.
\end{lemma}
\begin{proof}
	If $\sigma$ is a maximal simplex of dimension $k$ in a simplicial complex then \[h_k(\sigma)= H_k(X, X-\ostar(\sigma))=R,\] so every maximal simplex is $n$-dimensional. 
\end{proof}

\begin{remark}
	Lemma~\ref{pureremark} means we do not have to distinguish between ``locally CM of dimension $n$'' and ``$n$-dimensional and locally CM'' for a simplicial complex $X$. However, the condition that $X$ be locally CM of dimension $n$ at a subcomplex $L$ does not imply that $X$ is $n$-dimensional, and in what follows we will often need to separately assume that $X$ is $n$-dimensional.
\end{remark}

\begin{remark}
	The local homology at a simplex $\sigma$ is the same as the local homology of $X$ at a point $p$ in the interior of $\sigma$. It follows that the CM property is independent of a given triangulation of $X$ (see \cite{Munkres1984} for more details). However, it is generally badly behaved under homotopy (e.g. adding a hair to a manifold of dimension at least $2$).
\end{remark}
\begin{remark}
	Bredon's book on sheaf theory \cite[Chapter II.9]{Bredon2012} has the notion of a \emph{weak homology manifold}. This is a version of the locally CM condition that applies to a larger family of locally compact spaces.
\end{remark}

\begin{remark}
	Quillen \cite{Quillen1978} introduced a stronger version of the CM condition that requires links of simplices to be homotopy equivalent to wedges of spheres. We will refer to this condition as \emph{homotopy} CM. Any homotopy CM complex is locally CM in our sense, see Lemma~\ref{lem:munkres} below.
\end{remark}

\subsubsection{Examples}\label{cmexsect}
CM complexes are ubiquitous.
\begin{itemize}
	\item Any triangulated manifold is locally CM.
	\item Any pure 1-dimensional simplicial complex is locally CM, these are just graphs in which every vertex has an adjacent edge. It is CM if it is additionally connected.
	\item A pure 2-dimensional simplicial complex is CM if and only if the link of every vertex is connected. 
	\item Any shellable complex is homotopy CM \cite{Bjoerner1984}.
	\item Spherical buildings are homotopy CM \cite{Bjoerner1984}. Many other coset complexes are also homotopy CM \cite{Quillen1978,Brueck2020}.
	\item The spine of Culler-Vogtmann Outer space is CM \cite{Vogtmann1990}.
	\item The Salvetti complex of a right-angled Artin group (RAAG) is locally CM if and only if the RAAG is a duality group \cite{Brady2001}.
\end{itemize}
For surveys on CM complexes from the point of view of combinatorial topology see \cite{Bjoerner2016, Hochster}.

\subsubsection{Characterisation in terms of links}\label{linkcharsect}
Other texts use a characterisation of CM in terms of the homology of the links of simplices, related to local homology via:

\begin{lemma}[{\cite[Lemma 3.3]{Munkres1984}}]\label{lem:munkres}
	Let $\smallsimplex \in X_l$ be an $l$-simplex of a simplicial complex $X$. Then
	$$
	h_i(\smallsimplex)\cong \tilde{H}_{i-l-1}(\lk \smallsimplex).
	$$
\end{lemma}

\begin{remark}
	It is often easier to compute the homology of links than the relative homology at stars (for example Salvetti complexes for RAAGs have links that can be described in terms of their defining graph). However, this has the drawback that there is no direct map $\tilde{H}_*(\lk \smallsimplex)\rar \tilde{H}_{*}(\lk \smallsimplexone)$ for $\smallsimplex<\smallsimplexone$, obscuring the restriction maps of the sheaf $h_*$.
\end{remark}

\subsection{Properties of locally CM complexes}\label{cmpropssection}
When an $n$-dimensional complex $X$ is locally CM at $L$ over a PID $R$, the Universal Coefficient Theorem tells us that for each simplex $\smallsimplex\in L$
\begin{equation}
h^*(\smallsimplex)\cong\hom(h_*(\smallsimplex),R).
\end{equation}
If $R$ is not a PID this statement still holds -- we prove this in \textsection\ref{uctcmsection}. There, we also show that one can further drop the assumption that $X$ is $n$-dimensional.

When $R$ is hereditary, this allows us to give interpretations of the degree zero homology of $h^*$ in terms of homomorphisms on the sections (see \textsection\ref{sectionsection}) of $h_*$. For locally finite homology this is fairly straightforward: we discuss this in \textsection\ref{locfindualsect}. Characterising the non-locally finite homology of $h^*$ is more involved, and is done in terms of \emph{compactly determined homomorphisms} on sections of $h_*$. This characterisation requires a semistability condition on the sections of $h_*$: we discuss compactly determined homomorphisms and this semistability condition for an arbitrary sheaf in \textsection\ref{compactsemistabsect}, and apply this to the local (co)homology in \textsection\ref{homloccohcharsect}. 

\subsubsection{A universal coefficient theorem for local (co)homology of CM complexes}\label{uctcmsection}
Recall (from \textsection\ref{s:modules}) that for bimodules $M$ and $N$ we denote by $\hom(M,N)$ the bimodule of left $R$-linear morphisms. We will show the following:
\begin{prop}\label{literallydual}
	Let $X$ be a simplicial complex that is locally CM of dimension $n$ at a subcomplex $L$. Then for every $\smallsimplex\in L$ we have a $R$-bimodule isomorphism
	$$
	h^n(\smallsimplex)\cong\hom(h_n(\smallsimplex),R).
	$$
	If $\smallsimplex<\smallsimplexone$, then $\hom(-,R)$ takes $h_n(\smallsimplex<\smallsimplexone)$ to $h^n(\smallsimplexone>\smallsimplex)$ under these isomorphisms.
 \end{prop}
This is an application of the following well-known generalisation of the Universal Coefficient Theorem. For details on this and spectral sequences coming from double complexes see e.g. \cite[Appendix A]{Bredon2012}, \cite[Chapter 7]{Brown1982}, or \cite[Section 5.6]{Weibel1995}.

\begin{theorem}[Universal Coefficient Spectral Sequence]\label{ucss}
	Let $C_\bullet$ be a chain complex of projective $R$-modules concentrated in non-negative degrees and $M$ be an $R$-module. Then there is a convergent spectral sequence:
	$$
	E_2^{p,q}=\mathbf{Ext}_R^q(H_p(C_\bullet),M) \Rightarrow H^{p+q}(\hom(C_\bullet, M)).
	$$
\end{theorem}

\begin{proof}[Proof of Proposition~\ref{literallydual}]
	Let  $\smallsimplex\in L$ be a simplex. As $X$ is locally finite $C_\bullet(X,X-\ostar\smallsimplex)$ is a chain complex of finitely generated free $R$-modules, see \textsection\ref{lochomcompsection}. Therefore we can apply Theorem~\ref{ucss} to $C_\bullet=C_\bullet(X,X-\ostar\smallsimplex)$ and $M=R$. By assumption $X$ is locally CM at $L$ of dimension $n$, so we have that $H_p(C_\bullet)=h_p(\smallsimplex)$ is a free $R$-module in degree $n$ and vanishes otherwise. This means that $\mathbf{Ext}_R^q(H_p(C_\bullet),R)$ vanishes for $q\neq 0$ and $p\neq n$, and the spectral sequence collapses on the $E_2$-page with only nonvanishing group
	$$
	E^{n,0}_{2}= \mathbf{Ext}_R^0(h_n(\smallsimplex),R)=\hom(h_n(\smallsimplex),R).
	$$
	The spectral sequence converges to the cohomology of $\hom(C_\bullet(X,X-\ostar\smallsimplex),R)\cong C^\bullet(X,X-\ostar\smallsimplex)$. Combined with the fact that the spectral sequence collapses, this gives
	$$
	h^{n}(\smallsimplex)\cong\hom(h_n(\smallsimplex),R).
	$$
	The spectral sequence from Theorem~\ref{ucss} is natural in $C_\bullet$, so $\hom(-,R)$ takes $h_n(\smallsimplex<\smallsimplexone)$ to $h^n(\smallsimplexone>\smallsimplex)$ under this isomorphism.\end{proof}

\subsubsection{Locally finite homology in degree zero}\label{locfindualsect}
For the remainder of Section~\ref{cmpropssection} we will assume that $R$ is a hereditary ring. We have the following characterisation of the locally finite homology in degree zero of the local cohomology cosheaf:

\begin{proposition}\label{cosheafhomprop}
	Let $X$ be $n$-dimensional and locally CM of dimension $n$ over a hereditary ring $R$ at a subcomplex $L$. We have an isomorphism of $R$-bimodules
	\begin{equation}
	H^\lf_0(L;h^n|_L)\xrightarrow{\cong} \hom(\Gamma_c(h_n|_L),R),\label{lfcohshhomsecthom}
	\end{equation}
	natural with respect to inclusions of subcomplexes at which $X$ is locally CM. This map is at the chain level on a generator $v \under \bigsimplex^* \in C_0(L; h^n|_L)$ given by
	\begin{equation*}
		 v \under \bigsimplex^* \mapsto \left(s \mapsto  s_\bigsimplex^v\right),	
	\end{equation*}
	where $s^v_\bigsimplex$ denotes the $v\under \bigsimplex$-component of $s^v \in h_n(v)$. 
\end{proposition}
	\begin{proof}
		Applying $\hom(-,R)$ to the augmented cochain complex for the compactly supported cohomology of $h_n$ over $L$ (Equation \eqref{augmentedcompact})
		$$
		0\rar \Gamma_c(h_n|_L) \xrightarrow{\epsilon} C_c^0(L, h_n|_L) \rar C_c^1(L;h_n|_L)\rar \dots
		$$
		gives the augmented chain complex
		\begin{equation}\label{augmentedlf}
			\dots \rar C_1^\lf(L;h^n|_L) \xrightarrow{d_0} C_0^\lf(L;h^n|_L) \xrightarrow{\epsilon^*} \hom(\Gamma_c(h_n|_L),R) \rar 0.		
		\end{equation}
		Here we used that $h^n|_L\cong \hom(h_n|_L,R)$ by Proposition~\ref{literallydual}. To see that the last map is indeed surjective, note that as $\im d_0$ is a submodule of a projective module and $R$ is hereditary
		$$
		0 \rar \Gamma_c(h_n|_L) \xrightarrow{\epsilon} C_c^0(L, h_n|_L) \rar \im d_0 \rar 0
		$$
		is a short exact sequence of projective $R$-modules. Hence $\epsilon$ admits a splitting and so does its dual $\epsilon^*$ in Equation~\eqref{augmentedlf}. On cohomology $\epsilon^*$ therefore induces the isomorphism from Equation~\eqref{lfcohshhomsecthom}. One checks it is indeed given by the above formula. This formula also makes it clear that this isomorphism is natural.
	\end{proof}

\subsubsection{Compactly determined morphisms and semistability}\label{compactsemistabsect}
To characterise the (non-locally finite) cosheaf homology in degree zero, we use the following definition:
	
	\begin{definition}\label{compactlysuphomsdef}
		Let $X$ be a complex, let $\mathcal{F}$ be a sheaf on $X$, and let $M$ be an $R$-module. A left $R$-module morphism $f\colon \Gamma(\mathcal{F})\rightarrow M$ is called \emph{compactly determined} if there exists a finite subcomplex $K$ of $X$ such that for all sections $s \in \Gamma(X;\mathcal{F})$ with $s|_L=0$ one has $f(s)=0$. The \emph{right module of compactly determined morphisms} will be denoted $\homcd(\Gamma(\mathcal{F}),M)$.
	\end{definition}

Under a semistability condition on $\mathcal{F}$ these compactly determined morphisms come up naturally as dual to a colimit of sections over a filtration of $X$ by finite complexes. Recall that an inverse system $\{\Gamma_i\}_{i \in \N}$ is called \emph{semi-stable} if for every $i$ there is a $j\geq i$ such that for all $k\geq j$ the images of the maps $\Gamma_k\rar \Gamma_i$ agree with the image of the map $\Gamma_{j}\rar \Gamma_i$. We need the following consequence of semistability for an inverse system of projective modules over a hereditary ring.

\begin{lemma}\label{semistablesplit}
	Let $R$ be a hereditary ring, let $\{\Gamma_i\}_{i \in \N}$ be a semistable inverse system of projective $R$-modules and write $\Gamma=\lim_i \Gamma_i$. Then for all large enough $k$ the module $\Gamma_k$ splits as a direct sum
	$$
	\Gamma_k\cong I \oplus Z_k,
	$$
	where $I$ is isomorphic to a submodule of $\Gamma$ (independent of $k$) along the canonical map $r_k \colon \Gamma\rar \Gamma_k$ and $Z_k=\ker(\Gamma_k \rar \Gamma_i)$ for some $i$.
\end{lemma}

\begin{proof}
	By semistability of $\{\Gamma_i\}_{i\in \N}$ there exists $j\geq i$ such that for all $k\geq j$ the restriction maps from $\Gamma_j$ and $\Gamma_k$ to $\Gamma_i$ have the same image $I$. Because the $\Gamma_k$ are projective and $R$ is hereditary, $I$ is projective. So we get for each $k\geq j$ a direct sum decomposition
	$$
	\Gamma_k\cong I \oplus Z_k
	$$
	where $Z_k=\ker(\Gamma_k\rar \Gamma_i)$. This allows us to define a system of maps $\{I \rar \Gamma_k \}_{k\geq j}$ which extends to a system of maps into the diagram $\{\Gamma_i\}_{i\in \N}$ by composing with the maps in the inverse system. By the universal property this gives a morphism $i\colon I \rar \Gamma$, which for $k\geq j$ and $s \in I \subset \Gamma_k$ satisfies $r_k i(s)=s$, exhibiting $I$ as a submodule of $\Gamma$.
\end{proof}

This has a useful consequence for the colimit of the dual system. Write $M^\vee$ for the $R$-dual $\hom(M,R)$ of a left module $M$, viewed as a right $R$-module.
\begin{lemma}\label{semistableinjective}
	Let $\{\Gamma_i\}_{i \in \N}$ be a semistable inverse system of projective left modules over a hereditary ring $R$ and write $\Gamma=\lim_i \Gamma_i$. Then the canonical map
	$$
	\rho\colon \colim_i \Gamma_i^\vee \rar \Gamma^\vee
	$$
	is injective.
\end{lemma}
\begin{proof}
	To prove that $\rho$ is injective, suppose that $\phi\in \ker \rho$ and pick $\phi_i \in \Gamma_i^\vee$ representing $\phi$. We want to show that $\phi_i$ represents the trivial element in the colimit, for this it suffices to show that there exists $k>i$ such that $\phi_i r_i^k=0$, where $r_i^k\colon \Gamma_k \rar \Gamma_i$ is the restriction map. Picking $k$ large enough, we can apply Lemma~\ref{semistablesplit} and evaluate $\phi_ir_i^k$ on $\Gamma_k \cong I \oplus Z_k$ by evaluating on elements in $I$ and $Z_k$. By definition, $\phi_ir^k_i$ is zero on the $Z_k$ summand, while for $s\in I$ we have $r^k_i(s)= r_i i(s)$, where $i\colon I \rar \Gamma$ exhibits $I$ as a submodule of $\Gamma$. But this means
	$$
	\phi_ir^k_i(s)= \phi_i r_i i(s)=\rho(\phi)(i(s))=0,
	$$ 
	so $\phi=0$.
\end{proof}

For sheaves of projective modules over a hereditary ring we additionally have:

\begin{lemma}\label{colimhomcdlem}
	Let $\{K_i\}_{i\in \N}$ be a filtration of $X$ by finite subcomplexes such that the inverse system $\{\Gamma(\mathcal{F}|_{K_i})\}_{i\in \N}$ of sections of a sheaf $\mathcal{F}$ of projective left modules over a hereditary ring on $X$ is semistable. Then we have that the map
	$$
	\colim_i\hom(\Gamma({\mathcal{F}|_{K_i}}), R)\xrightarrow{\rho}\homcd (\Gamma(\mathcal{F}),R),
	$$
	induced by the universal property for the system of maps given by pulling back along the restriction maps $\Gamma(\mathcal{F})\rar\Gamma({\mathcal{F}|_{K_i}})$, is an isomorphism of right $R$-modules.
\end{lemma}
\begin{proof}
	The universal property induces a map to $\hom (\Gamma(\mathcal{F}),R)$, the dual of $\Gamma(\mathcal{F})$. Note that this map factors through $\homcd(\Gamma(\mathcal{F}),R)$ by construction. We need to show that it is an isomorphism onto this submodule. 
	
	Write $\Gamma_i=\Gamma({\mathcal{F}|_{K_i}})$ and $\Gamma=\Gamma(\mathcal{F})$ and $r_i\colon \Gamma\rar \Gamma_i$ for the restriction maps. Note that the $\Gamma_i$ are finite direct sums of projective modules, hence projective, so by Lemma~\ref{semistableinjective} the map $\rho$ is injective. We still need to show that $\rho$ surjects onto the compactly determined morphisms. So let $f\in \homcd(\Gamma,R)$, and suppose that $f$ is compactly determined on $K_i$, so that $f(s)=0$ whenever $r_i(s)=0$. We want to show that $f$ can be obtained by restriction to a finite subcomplex. By Lemma~\ref{semistablesplit}, there exists a finite subcomplex $K_k \supseteq K_i$ such that $\Gamma_k\cong I \oplus Z_k$ where $Z_k=\ker(\Gamma_k \rar \Gamma_i)$. Define $f'\in \Gamma_k^\vee$ by setting $f'|_{Z_k}=0$ and $f'|_{I}=f\circ i$, where $i\colon I \rar \Gamma$ is the map that exhibits $I$ as a submodule of $\Gamma$ satisfying $i r_k = \id_I$. We claim that $f'$ is such that precomposing $f'$ with the restriction $r_k\colon\Gamma \rar\Gamma_k$ is equal to $f$. Note that $r_k$ and $i$ split $I$ as a direct summand of $\Gamma$, with $\ker(r_k)$ as the complement. For $s\in \ker(r_k)$ we have $r_i(s)=r_kr_k^i(s)=0$, so $f(s)=0$ as $f$ is compactly determined on $K_i$. For $s\in I$ we have 
	$$
	f(s)=f(ir_k(s))=f'r_k(s)
	$$
	as claimed.
\end{proof}

\subsubsection{Characterising $H_0(L;h^n|_L)$ under semistability assumptions}\label{homloccohcharsect}
Applying the preceding discussion to the local homology sheaf gives the following analogue of Proposition~\ref{cosheafhomprop}:

\begin{theorem}\label{cosheafhomcompactsup}
	Let $R$ be a hereditary ring, and assume that an $n$-dimensional locally finite complex $X$ is locally CM of dimension $n$ over $R$ at a subcomplex $L$. If $\{K_i\}_{i\in \N}$ is a filtration of $L$ by finite subcomplexes such that the inverse system $\{\Gamma(\hres{K_i})\}_{i\in \N}$ is semistable we have an isomorphism of $R$-bimodules
	\begin{equation}
	H_0(L;h^n|_L)\xrightarrow{\cong} \homcd(\Gamma(h_n|_L),R).	\label{cohshhomcompsup}
	\end{equation}
	which is at the chain level on a generator $v \under \bigsimplex^* \in C_0(L; h^n|_L)$ given by
	\begin{equation*}
		v \under \bigsimplex^* \mapsto \left(s \mapsto  s_\bigsimplex^v\right).
	\end{equation*}
\end{theorem}
	\begin{proof}
	Observe that taking a colimit over the filtration by finite subcomplexes $K_i$ of $L$ we have, using that homology commutes with filtered colimits,
	$$
	H_0(L;h^n|_L)\cong \colim_i H_0(K_i;h^n|_{K_i})= \colim_i H^\lf_0(K_i;h^n|_{K_{i}}).
	$$
	Noting that $\Gamma_c(h_n|_{K_i})=\Gamma(h_n|_{K_i})$ for finite $K_i$, we then have by Proposition~\ref{cosheafhomprop}
	\begin{equation}\label{hiscolim}
	H_0(L;h^n|_L)\cong\colim_i \hom(\Gamma_c({h_n|_{K_i}}), R)=\colim_i \hom(\Gamma({h_n|_{K_i}}), R).	
	\end{equation}
	Noting that $h_n$ is a sheaf of projective left modules (Lemma~\ref{hnprojectivelemma}), by Lemma~\ref{colimhomcdlem} we then have
	$$
		H_0(L;h^n|_L)\cong \homcd (\Gamma(h_n|_L),R).
	$$
	The chain level formula makes it clear that this is an isomorphism of bimodules.
\end{proof}

\begin{remark}
	We expect that the requirement for $X$ to be locally CM in Proposition~\ref{cosheafhomprop} and Theorem~\ref{cosheafhomcompactsup} can be replaced with the weaker assumption that $h_*$ is a projective left module in every degree. One possible approach would be to apply the universal coefficient spectral sequence and the sequence of arguments above degree by degree. 
\end{remark}

When $R$ is a PID we have a consequence of this and results from \cite{Geoghegan2008}:
\begin{corollary}\label{c:semistable}
	Let $R$ be a PID and assume that an $n$-dimensional locally finite complex $X$ is locally CM of dimension $n$ over $R$ at a countable subcomplex $L$. Then for any filtration $\{K_i\}_{i\in \N}$ of $L$ by finite subcomplexes the inverse system $\{\Gamma(\hres{K_i})\}_{i\in \N}$ is semistable if and only if $H_0(L;h^*|_L)$ is a (countably generated) free $R$-module. In this case, we have
	$$
	H_0(L;h^n|_L)\cong \homcd (\Gamma(\mathcal{F}),R).
	$$
\end{corollary}

\begin{proof}
	Like in the proof of Theorem~\ref{cosheafhomcompactsup}, we have 
	$$
	H_0(L;h^n|_L)\cong \colim_i \hom(\Gamma(\hres{K_i}),R).
	$$
	Note that this module is countably generated as $L$ is countable and $X$ is locally finite. The result \cite[Proposition 12.5.6]{Geoghegan2008} states that $\{\Gamma(\hres{K_i})\}_{i\in \N}$ is semistable if and only if $\colim_i \hom(\Gamma(\hres{K_i}),R)$ is a countably generated free module, establishing the equivalence claimed.	In the case that either of these equivalent conditions holds, Lemma~\ref{colimhomcdlem} gives the isomorphism.
\end{proof}

\section{Duality for Cohen--Macaulay Complexes}\label{cmdualitysect}
In this section we prove the CM Duality Theorem. In Section~\ref{mvsect} we set up, for a full subcomplex $L$ of a locally finite oriented simplicial complex $X$, a \emph{Mayer--Vietoris double complex $D^\bullet_\bullet(L)$} such that the homology of the total complex of $D^\bullet_\bullet(L)$ is the homology $H_l(X,L\vc)$ of $X$ relative to the vertex complement $L\vc$ (Definition~\ref{vcandbeforedef}) of $L$. When $X$ is locally CM of dimension $n$ at $L$, the column spectral sequence of $D^\bullet_\bullet(L)$ collapses and gives an isomorphism $H_{n-p}(X,L\vc)\xrightarrow{\cong} H_c^p(X; h_n|_L)$. In Theorem~\ref{standardexactrowspecseq} we provide a chain level expression for the inverse to this isomorphism, which is reminiscent of the cap product. This is used to establish CM Duality Theorem in Section~\ref{dualitysection}.

\subsection{The Mayer--Vietoris spectral sequence}\label{mvsect}
In this section we define the Mayer--Vietoris (MV) double complex $D^\bullet_\bullet(L)$ for a subcomplex $L$ of $X$ in \textsection\ref{mvdoublecompsect}, and discuss its column spectral sequence in \textsection\ref{mvcolmnspecsecsec}. We then prove that the homology of $D^\bullet_\bullet(L)$ is isomorphic to $H_l(X,L\vc)$ by providing an augmentation for $D^\bullet_\star(L)$ by $C_\star(X,L\vc)$ and proving the resulting rows are acyclic in \textsection\ref{augmvsec}. The main technical tool is a map $\lambda$ called the \emph{last vertex lift} (Definition~\ref{lastvertexliftdef}). Using $\lambda$, we provide an explicit chain homotopy equivalence $\mathfrak{c}\colon \textbf{Tot}D\rar C_*(X,L\vc)$ in \textsection\ref{explicitinversesect}. This $\dcap$ will be compared to the cap product later. Variants of the MV double complex are needed to prove duality results involving other versions of the cap product discussed in Proposition~\ref{usefulrelativecups}. These are discussed in \textsection\ref{prodmvcompsec} and \textsection\ref{dualsection}.

\subsubsection{The Mayer--Vietoris double complex}\label{mvdoublecompsect}
For a full subcomplex $L$ of a locally finite oriented simplicial complex $X$ we define the \emph{Mayer--Vietoris double complex} of $R$-bimodules $D_{\bullet}^{\bullet}(L)$ by setting 
\begin{equation}\label{ddef}
D^l_k(L)=\bigoplus_{\smallsimplex \in L_l} C_k(X,X-\ostar \smallsimplex).	
\end{equation}
We define the differentials below. Figure~\ref{mvdoublecomplexfigure} sketches this double complex and relevant maps, some of which will be defined in later sections. In what follows we will suppress $L$ from the notation and just write $D_k^l$. Cf. Notation~\ref{undernotation} we denote the generators of $C_k(X,X-\ostar \smallsimplex)= C_k(X)/C_{k}(X-\ostar \smallsimplex)$ by $\smallsimplex\under\bigsimplex$ with $\bigsimplex\in X_k$ containing $\smallsimplex \in L_l$. We use the convention that $\smallsimplex\under\bigsimplex=0$ if $\bigsimplex$ does not contain $\smallsimplex$, and $\smallsimplex\under \bigsimplex=0$ in $D^l_k$ if $\smallsimplex\notin L_l$. For $\shcochain \in D^l_k$ we write, with coefficients $\shcochain_\bigsimplex^\smallsimplex\in R$,
$$
\shcochain = \sum_{\smallsimplex \in L_l}\sum_{\bigsimplex\geq \smallsimplex} \shcochain_\bigsimplex^\smallsimplex\, \smallsimplex\under\bigsimplex.
$$
The grading on $D_{\bullet}^{\bullet}$ is cohomological in the horizontal direction and homological in the vertical direction, this is also reflected in the position of the indices.

\begin{figure}[h]\label{mvdoublecomplexfigure}
	\begin{tikzcd}[column sep=small, row sep=tiny]
		0			&\dots 	&0				& 0			&0			&		& 0		& 0		\\
		D^0_n 		&\dots 	&D^{l-1}_n 		& D^l_n 	&D^{l+1}_n	&\dots	&D^n_d	&0		\\
		\vdots		&		&\vdots			&\vdots		&\vdots		&		&\vdots	&\vdots	\\
		D^0_k 		&\dots	&D^{l-1}_k		&D^l_k\arrow[d,"d_k"]\arrow[r,"i^l"]\arrow[l,"\lambda"']\arrow[ld,"\Lambda"]	&D^{l+1}_k	&\dots	& D^d_k	&	0\\
		D^0_{k-1}	&\dots	&D^{l-1}_{k-1}	&D^l_{k-1}	&\dots		&		&	\vdots	&\vdots	\\
		\vdots		&		&\vdots			&\vdots		&			&		&D^d_{d}	&0		\\
		&		&				&			&			&\iddots	&\iddots&		\\
		\vdots		&		&				&\vdots		&			&\iddots	&		&		\\
		D^0_{l}		&		&				&D^l_l		&	0		&		&		&		\\
		\vdots		&		&		\iddots	&	0		&			&		&		&		\\
		\vdots		&\iddots	&	\iddots		&			&			&		&		&		\\		
		D^0_0		&\iddots&				&			&			&		&		&		\\
		0			&		&				&			&			&		&		&
	\end{tikzcd}	
	\caption{The MV double complex $D(L)$ of a $d$ dimensional subcomplex $L$ of an $n$-dimensional complex $X$ with the differentials $i^l$ and $d_k$, the last vertex lift $\lambda$, and the diagonal shift $\Lambda$.}
\end{figure}

The horizontal differential $i^{l}\colon D^l_k \rightarrow D^{l+1}_k$ is induced by the inclusions $(X,X-\ostar \smallsimplex_{\langle j\rangle})\hookrightarrow (X,X-\ostar \smallsimplex$), where $\smallsimplex_{\langle j\rangle}$ denotes the $j$th face of $\smallsimplex$. Similarly to the differential for sheaf cohomology from Equation~\eqref{sheafboundary}, it is given by
$$
i^{l}\shcochain= \sum_{\smallsimplex \in L_{l+1}}\sum_{\bigsimplex\geq \smallsimplex}\sum_{j=0}^{l+1}(-1)^j \shcochain_\bigsimplex^{\smallsimplex_{\langle j \rangle}} \smallsimplex\under\bigsimplex.
$$
Note for future use that the coefficient of $\smallsimplex\under\bigsimplex$ in $i^l\shcochain$ is
\begin{equation}\label{horizontaldiffcoef}
	\left(i^l\shcochain\right)_\bigsimplex^\smallsimplex=\sum_{j=0}^{l+1}(-1)^j\shcochain^{\smallsimplex_{\langle j\rangle}}_\bigsimplex.
\end{equation}

The vertical differentials $d_k\colon D^l_k \to D^l_{k-1}$ are direct sums of the differentials on $C_*(X,X-\ostar \smallsimplex)$. For $\shcochain\in D^l_k$ this means
$$
d_k \shcochain= \sum_{\smallsimplex \in L_l}\sum_{\bigsimplex\in X_k} \sum_{j=0}^{k}(-1)^{j}\shcochain^\smallsimplex_{\bigsimplex}\smallsimplex \under \bigsimplex_{\langle j \rangle},
$$
so the coefficient of $\smallsimplex \under\bigsimplex$ for $\bigsimplex \in X_{k-1}$ in the differential is
\begin{equation}\label{vertdiffcoefs}
	(d_k \shcochain)^\smallsimplex_\bigsimplex= \sum_{j=0}^{k}\sum_{\bigsimplexone_{\langle j \rangle}=\bigsimplex}(-1)^{j}\shcochain^\smallsimplex_{\bigsimplexone}.
\end{equation}

As the horizontal differentials are induced by maps of spaces they commute with the vertical differentials, so $D^\bullet_\bullet$ is indeed a double complex. The associated \emph{total complex} $\mathbf{Tot}D_\bullet$ is the chain complex of $R$-bimodules with chain modules $\mathbf{Tot}D_j$ obtained by summing over the diagonals $k-l=j$. We equip this with the differential that on the $D^l_k$ summand is given by
\begin{equation}\label{totaldifferentialD}
	d^{D,l}_k=(-1)^{k-l}d_k+ i^l.	
\end{equation}
Here $(-1)^l$ is the usual sign in the differential of a total complex. The $(-1)^k$ is a non-standard convention we use here to circumvent sign issues, and corresponds to the sign in the differential in Proposition~\ref{capischainmap}. It does not change the homology of the total complex.

\begin{remark}
	The Mayer--Vietoris double complex is a variant of the double complex that gives rise to the Mayer--Vietoris spectral sequence (see \cite[Chapter II.8]{Tu1982}).
\end{remark}

We need a standard fact about the kernels of differentials in double complexes:
\begin{lemma}\label{kerisubcomplex}
	Fix $l$. The kernel $\ker i^l$ of $i^l$ equipped with the restriction of the vertical differentials is a subcomplex of $\mathbf{Tot}D$.
\end{lemma}

\subsubsection{The column spectral sequence of the MV double complex}\label{mvcolmnspecsecsec}
Associated to a double complex are two spectral sequences, both of which abut to the homology of the total complex.
We call the spectral sequence obtained by taking the vertical differentials of $D^\bullet_\bullet$ first the \emph{column spectral sequence} $\mathbf{Col}D$. The $E^1$-page of this is
\begin{equation}\label{E1D}
\mathbf{Col}D^{1,l}_{k}= \bigoplus_{\smallsimplex \in L_l} h^X_k(\smallsimplex),	
\end{equation}
where $h^X_k$ denotes the degree $k$ part of the local homology sheaf of $X$. The differentials on the $E^1$-page are exactly the sheaf cohomology differentials from Equation~\eqref{sheafcohomchains}, so the $E^2$-page is
$$
\mathbf{Col}D^{2,l}_{k}= H_c^l(L; h^X_k|_L).
$$

The following observation makes it clear why this double complex is relevant:

\begin{prop}\label{CMspecseqcollapse}
	If $X$ is $n$-dimensional locally Cohen--Macaulay at $L$ the spectral sequence $\mathbf{Col}D(L)$ collapses on the $E^2$-page to give
	$$
	H_{n-l}(\mathbf{Tot}D(L))\cong H_c^l(L; h^X_n|_L).
	$$
\end{prop}

\begin{proof}
	Note that if $X$ is locally CM at $L$, only the $n$th row survives to the $E_1$-page, and this row is exactly the cochain complex that computes $H_c^l(L; h^X_n|_L)$.
\end{proof}

\subsubsection{The augmented Mayer--Vietoris complex and the last vertex lift}\label{augmvsec}
We will now show that the homology of $\mathbf{Tot}D_\bullet$ is isomorphic to $H_*(X,L\vc)$ (regardless of whether $X$ is locally CM). For each $k$ the rows of the double complex $D^\bullet_k$ admit augmentations by $C_{k}(X,L^\textnormal{vc})$ that are compatible with the vertical differentials:

\begin{proposition}\label{doubleexactrows}
	Let $L$ be a full subcomplex of a locally finite oriented simplicial complex $X$ with $L\vc$ ordered before $L$ (see Definition~\ref{vcandbeforedef}). Let $D^l_\bullet$ denote the $l$-th column of $D^\bullet_\bullet(L)$ equipped with the vertical differentials $d_k$. Then the following sequence of chain complexes is exact:
	$$
	0\rightarrow C_\bullet(X,L^\textnormal{vc}) \xrightarrow{\epsilon} D_\bullet^0 \xrightarrow{i^0} D_\bullet^1 \xrightarrow{i^1} \cdots,
	$$
	where $\epsilon$ is the product of the maps $C_{\bullet}(X,L^\textnormal{vc})\rar C_{\bullet}(X,X-\ostar v)$ of chain complexes induced by the inclusions of pairs $(X,L^\textnormal{vc})\rar (X,X-\ostar v)$ for all $v\in L_0$. 
\end{proposition}

Note that as $\epsilon$ is induced by maps of spaces it is a chain map. In the proof of this proposition, we will use an explicit splitting of the differential $i^l$. We will also make use of this splitting in finding our explicit inverse to the augmentation $\epsilon$ in \textsection\ref{explicitinversesect}. It is defined as follows:

\begin{definition}\label{lastvertexliftdef}
	Let $L\subset X$ be a subcomplex so that $L\vc$ is before $L$. Then the \emph{last-vertex lift} $\lambda\shcochain\in D_k^{l-1}(L)$ of $\shcochain\in D_k^l(L)$ is the chain with coefficients 
	\begin{equation}\label{lastvertexlift}
		(\lambda\shcochain)^\smallsimplex_\bigsimplex=(-1)^l\shcochain^{\smallsimplex\star \bigsimplex_{k}}_\bigsimplex. 
	\end{equation}
	Declaring $\lambda \shcochain=0$ if $\shcochain \in D^0_k$, this defines a degree -1 map 
	\begin{align*}
		\lambda\colon &D^{l}_k \rar D^{l-1}_k\\	
					& \shcochain \mapsto \lambda \shcochain.
	\end{align*}
\end{definition}

On a generator $\smallsimplex \under \bigsimplex \in D^l_k$, it is a worthwhile exercise to check that
\begin{equation}\label{lambdaongen}
\lambda(\smallsimplex\under \bigsimplex)= \begin{cases}
	(-1)^l\smallsimplex_{\leq l-1} \under \bigsimplex & \mbox{ if }\smallsimplex_{l}=\bigsimplex_k\\
	0& \mbox{ otherwise.}
\end{cases}	
\end{equation}

For the proof of Proposition~\ref{doubleexactrows} the crucial property of $\lambda$ is:

\begin{lemma}\label{commilambda}
	For $l>0$ and $\phi \in D^l_k$ the last vertex lift satisfies
	\begin{equation}\label{icommutatator}
		\lambda i \shcochain + i \lambda \shcochain =\shcochain.
	\end{equation}	
	For $l=0$ and a generator $\bigsimplex_m\under \bigsimplex\in D^0_k$ we have
	$$
	\lambda i (\bigsimplex_m \under \bigsimplex)= \begin{cases}
		- \sum_{j=0}^{k-1}\bigsimplex_j\under \bigsimplex & \mbox{ for }m=k\\ \bigsimplex_m \under \alpha &\mbox{ otherwise.}
	\end{cases}
	$$
\end{lemma}
\begin{proof}
	This is a direct computation. For $l>0$, note that if $\shcochain\in D^l_k$, then $i\shcochain\in D^{l+1}_k$ and $\lambda \shcochain \in D^{l-1}_k$. Let $\smallsimplex\under \bigsimplex \in D^l_k$ be a generator. Then the component $(\lambda i \shcochain)^\smallsimplex_\bigsimplex $ of $\lambda i \shcochain$ for $\shcochain=\sum \phi^\smallsimplexone_\bigsimplexone \smallsimplexone \under \bigsimplexone$ at this generator is
	\begin{equation*}
		(\lambda i \shcochain )^\smallsimplex_\bigsimplex= (-1)^{l+1}(i \shcochain)^{\smallsimplex\star \bigsimplex_k}_\bigsimplex
															= \sum_{j=0}^{l+1}(-1)^{l+1+j} \shcochain_{\bigsimplex}^{(\smallsimplex \star \bigsimplex_k)_{\langle j \rangle}}
															= \sum_{j=0}^{l}(-1)^{j+l+1} \shcochain_{\bigsimplex}^{\smallsimplex_{\langle j \rangle} \star \bigsimplex_k}+\shcochain_{\bigsimplex}^{\smallsimplex },
	\end{equation*}
	where we used Equation~\eqref{horizontaldiffcoef} in the second equality, and the fact that the $l+1$th face of $\smallsimplex\star \bigsimplex_k$ is $\smallsimplex$ in the last. For $ i \lambda \shcochain$ we find that the corresponding coefficient is
	\begin{equation*}
		(i\lambda  \shcochain )^\smallsimplex_\bigsimplex	= \sum_{j=0}^{l}(-1)^j (\lambda \phi)^{\smallsimplex_{\langle j \rangle}}_\bigsimplex= \sum_{j=0}^{l}(-1)^{j+l} \phi^{\smallsimplex_{\langle j \rangle}\star \bigsimplex_k}_\bigsimplex.
	\end{equation*}
	Adding these together yields \eqref{icommutatator}.
	
	For $l=0$, let $\bigsimplex_m \under \bigsimplex \in D^0_k$. Using Equation~\eqref{horizontaldiffcoef} we see that only non-degenerate edges of $\bigsimplex$ contribute to the differential, with positive sign exactly when $\bigsimplex_m$ is the last vertex (equivalently zeroth face) of the edge. This gives:
	$$
	i(\bigsimplex_m \under \bigsimplex)= \sum_{j=0}^{m-1}(\bigsimplex_j \star \bigsimplex_m)\under \bigsimplex-\sum_{j=m+1}^k(\bigsimplex_m \star \bigsimplex_j)\under \bigsimplex. 
	$$
	Applying $\lambda$ to this gives
	\begin{align*}
		\lambda i (\bigsimplex_m \under \bigsimplex) 
		& = \begin{cases}
			- \sum_{j=0}^{k-1}\bigsimplex_j\under \bigsimplex & \mbox{ for }m=k\\ \bigsimplex_m \under \alpha &\mbox{ otherwise.}
		\end{cases}
	\end{align*}
	This finishes the proof.
\end{proof}

At the left end of the augmented chain complex we will use the map
\begin{align}\begin{split}\label{kappadef}
			\kappa\colon & D^0_k \rar C_k(X;L\vc)\\
			& \bigsimplex_m \under \bigsimplex \mapsto \delta_{m,k}\bigsimplex,
		\end{split}
\end{align}
where $\delta_{m,k}$ denotes the Kronecker delta. For this map we have:
\begin{lemma}\label{kappaepsilon}
	For $\kappa$ as defined above we have
	\begin{align*}
		\epsilon\kappa &=\id- \lambda i^0\\
		\kappa \epsilon&=\id,
	\end{align*}
	where $\epsilon$ is the augmentation map from Proposition~\ref{doubleexactrows}.
\end{lemma}

\begin{proof}
	This is another direct computation. The map $\epsilon$ gives on a $k$-simplex $\bigsimplex\in X_k$
	$$
		\epsilon(\bigsimplex)=\sum_{j=0}^{k}\bigsimplex_j\under\bigsimplex,
	$$
	where $\bigsimplex_j$ denotes the $j$th vertex of $\bigsimplex$. For a generator $\bigsimplex_m\under \bigsimplex$ we get
	$$
		\epsilon\kappa (\bigsimplex_m\under \bigsimplex)= \begin{cases}
		\sum_{j=0}^{k} \bigsimplex_j \under \bigsimplex & \mbox{ for }m=k\\ 0 &\mbox{ otherwise.}
	\end{cases}
	$$
	So we have, using Lemma~\ref{commilambda},
	\begin{align*}
	 	(\epsilon \kappa+ \lambda i^0)(\bigsimplex_m\under \bigsimplex)&=\begin{cases} \sum_{j=0}^{k} \bigsimplex_j \under \bigsimplex - \sum_{j=0}^{k-1} \bigsimplex_j \under \bigsimplex & \mbox{ for }m=k\\ 0+ \bigsimplex_m\under \bigsimplex &\mbox{ otherwise} \end{cases}\\ & = \bigsimplex_m \under \bigsimplex.
	\end{align*}
	For the second equality, on $\bigsimplex \in X_k-L_k$ we have 
	$$
	 \kappa \epsilon (\bigsimplex)= \kappa\left(\sum_{j=0}^{k}\bigsimplex_j\under \bigsimplex\right)= \bigsimplex.
	$$
	This finishes the proof.
\end{proof}

\begin{proof}[Proof of Proposition~\ref{doubleexactrows}]
	Fix $k$. We will show that the $k$th row is exact by showing that the identity on the augmented cochain complex 
	$$
	\tilde{C}^\bullet = 	0\rightarrow C_k(X,L^\textnormal{vc}) \xrightarrow{\epsilon} D_k^0 \xrightarrow{i^0} D_k^1 \xrightarrow{i^1} \cdots,
	$$
	is null homotopic. Write $\tilde{i}$ for the differential on the augmented chain complex. Our null-homotopy $\tilde{\lambda}\colon \tilde{C}\rar \tilde{C}[-1]$ is defined by setting $\tilde{\lambda}=\lambda$ on $D^l_k$ with $l>0$. On $D^0_l$ we set $\tilde{\lambda}=\kappa $, note that necessarily $\tilde{\lambda}=0$ on $C_k(X, L\vc)$. Then we have
	\begin{equation}\label{lambdanullhty}
	\shcochain=\tilde{\lambda} \tilde{i}\shcochain+ \tilde{i}\tilde{\lambda}\shcochain
	\end{equation}
	for any $\shcochain \in \tilde{C}^\bullet$. For $l>0$ this is Lemma~\ref{commilambda}, for $l=0$ this uses the first equation from Lemma~\ref{kappaepsilon}, and for $l=-1$ this is the second equation from Lemma~\ref{kappaepsilon}. 
\end{proof}	

We can deduce that $\epsilon$ is a chain map onto the subcomplex (Lemma~\ref{kerisubcomplex}) $\ker i^0\subset \mathbf{Tot} D$. The map $\kappa$ provides an inverse chain map in the other direction:
\begin{lemma}\label{kappachainmap}
	The map $\kappa$ from Equation~\eqref{kappadef} restricts to a chain map $$\kappa|_{\ker i^0}\colon \ker i^0\rar C_\bullet (X,L\vc).$$ inverse to $\epsilon$.
\end{lemma}	
\begin{proof}
	By Lemma~\ref{kappaepsilon} the maps $\kappa|_{\ker i^0}$ and $\epsilon$ are mutually inverse, and a module-wise inverse to a chain map is necessarily a chain map.
\end{proof}

\subsubsection{An explicit inverse to the augmentation map}\label{explicitinversesect} Proposition~\ref{doubleexactrows} implies that the augmented double complex $C_{\bullet}(X,L^\textnormal{vc})\xrightarrow{\epsilon}D_{\bullet}^{\bullet}$ is a first-quadrant double complex with exact rows, so $\epsilon$ induces an isomorphism in homology:	as the rows of $D_\bullet^\bullet$ are exact, the $E^1$-page of the row spectral sequence $\mathbf{Row}D$ is concentrated the zeroth column. Using the augmentation $\epsilon$ we see that this column is isomorphic to $C_\bullet(X, L^\textnormal{vc})$. The spectral sequence then collapses on the $E^2$-page to show that $\epsilon$ induces an isomorphism from $H_*(X,L^\textnormal{vc})$ to $H_*(\mathbf{Tot}D)$ (see \cite[Chapter II.8]{Tu1982}). 

We now show how to construct an explicit chain level inverse to $\epsilon$. We will later compare this to taking the cap product with a fundamental class. Before doing this, we need to pay the price\footnote{Compared to the sign issues avoided, we assure you that this is quite a small price.} for our non-standard sign convention in the differential (Equation~\eqref{totaldifferentialD}) on the total complex for $D^\bullet_\bullet$: we have effectively multiplied the vertical differential on each $k$th row of $D^\bullet_\bullet$ by $(-1)^k$, so in order for $\epsilon$ to give a chain map $C_{\bullet}(X,L^\textnormal{vc})\rar \mathbf{Tot}D_{\bullet}$, we should do the same to $C_{\bullet}(X,L^\textnormal{vc})$. We define $\bar{C}_{\bullet}(X,L^\textnormal{vc})$ by setting
$$
\bar{C}_{k}(X,L^\textnormal{vc})=C_{k}(X,L^\textnormal{vc}),
$$
with differential
\begin{align}
	\begin{split}\label{modifieddiff}
	\bar{d}\colon\bar{C}_{k}(X,L^\textnormal{vc})&\rar \bar{C}_{k-1}(X,L^\textnormal{vc})\\
	\chain &\mapsto (-1)^kd\chain,
	\end{split}
\end{align}
where $d$ denotes the usual differential on $C_{\bullet}(X,L^\textnormal{vc})$. Note that
$$
H_k\bar{C}_{\bullet}(X,L^\textnormal{vc})=H_k C_{\bullet}(X,L^\textnormal{vc})=H_{k}(X,L^\textnormal{vc}),
$$
for all $k$. With this definition, our next goal is:

\begin{theorem}\label{standardexactrowspecseq}
	Let $L$ be a full subcomplex of a locally finite, finite dimensional complex $X$ oriented so that $L\vc$ comes before $L$, and let $D^\bullet_\bullet(L)$ be as defined in Section~\ref{mvdoublecompsect}. Then there is a chain homotopy equivalence
	$$
	\dcap\colon \mathbf{Tot}D_\bullet \rar \bar{C}_{\bullet}(X,L^\textnormal{vc})
	$$
	defined summand-wise by
	\begin{alignat*}{4}
		\dcap^k_l \colon &&D_k^l(L) &\quad&\rightarrow 	&\quad&& C_{k-l}(X, L^\textnormal{vc})\\
					&& \shcochain 	&& \mapsto		&&& \sum_{\bigsimplex \in X_k} \shcochain_{ \bigsimplex}^{\bigsimplex_{\geq k- l}}\bigsimplex_{\leq k-l},
	\end{alignat*}
	with inverse the map $\epsilon\colon \bar{C}_{\bullet}(X,L^\textnormal{vc})\rightarrow \mathbf{Tot}D_{\bullet}$ from Proposition~\ref{doubleexactrows}.
	Furthermore $\dcap$ induces a natural (for inclusions of subcomplexes) $R$-bimodule isomorphism
	\begin{equation}\label{doublecomplexhomiso}
		H_*(\mathbf{Tot}D)\cong H_{*}(X,L^\textnormal{vc}).	
	\end{equation}
\end{theorem}

Note the similarity between $\dcap$ and the cap product (Equation~\eqref{cupversion1}). We split the proof of Theorem~\ref{standardexactrowspecseq} into the next two lemmas. The first expresses that isomorphism is natural:

\begin{lemma}\label{cisnatural}
	Let $K\subset K'$ be full subcomplexes of a locally finite oriented simplicial complex $X$. Then the inclusion $f\colon K\rar K'$ induces a map
	\begin{align*}
		f^\bullet_\bullet\colon& D^\bullet_\bullet(K) \rar 	D^\bullet_\bullet(K')\\
		& \smallsimplex \under \bigsimplex \mapsto \smallsimplex \under \bigsimplex.
	\end{align*}
	Denote by $\mathfrak{c}_K$ and $\mathfrak{c}_{K'}$ the isomorphism from Theorem~\ref{standardexactrowspecseq} for $K$ and $K'$ respectively. Then for the induced maps on homology we have
	$$
	(\mathfrak{c}_{K'})_* f^*_*=f^*_* (\mathfrak{c}_K)_*.
	$$
\end{lemma}
\begin{proof}
	The fact that $f^\bullet_\bullet$ is a chain map is a routine check. To see that $\mathfrak{c}$ is natural observe that, for $\epsilon_X$ and $\epsilon_Y$ the corresponding morphisms from Theorem~\ref{standardexactrowspecseq},
	$$
	\epsilon_Y f^\bullet_\bullet= f^\bullet_\bullet \epsilon_X
	$$
	already at the level of chains. The next lemma, Lemma~\ref{cishomotopyinverse}, shows that $\epsilon$ and $\mathfrak{c}$ are inverses for each other at the level of homology and then the statement follows.
\end{proof}

\begin{lemma}\label{cishomotopyinverse}
	The maps $\dcap$ and $\epsilon$ are each others homotopy inverses.
\end{lemma}
\begin{proof}
	In showing that $\dcap$ is a homotopy inverse for $\epsilon$, we in particular need to show that $\dcap$ is a chain map. We start with this. We claim that
	$$
	\dcap = \kappa \Lambda^m,
	$$
	where $\kappa\colon D^0_\bullet \rar C_\bullet(X;L\vc)$ is the map from Equation~\eqref{kappadef}, $m$ is a large enough integer, and $\Lambda$ is the \emph{diagonal shift}
	\begin{equation}\label{Lambdadef}
		\Lambda= \id - \lambda d^D - d^D\lambda,
	\end{equation}
	where  $d^{D,l}_k=(-1)^{k-l}d_k+ i^l$ is the differential on the total complex. The map $\Lambda\colon \mathbf{Tot}D\rar \mathbf{Tot}D$ is chain homotopic to the identity by construction. We show in Proposition~\ref{pythonlemma} that, for $m$ large enough, the chain map $\Lambda^m$ projects onto $\ker i^0$. We know that $\kappa|_{\ker i^0}$ is a chain map by Lemma~\ref{kappachainmap}, so $\dcap=\kappa\Lambda^m$ is a chain map. 
	
	To show that $\dcap = \kappa \Lambda^m$, note that on a generator $\smallsimplex \under \bigsimplex\in D^l_k$ the map $\dcap$ is
	\begin{equation}\label{congens}
		\dcap(\smallsimplex\under\bigsimplex)= \begin{cases}
			\bigsimplex_{\leq k-l}& \mbox{ if }\smallsimplex=\bigsimplex_{\geq k-l} \\ 0 & \mbox{ otherwise.}
		\end{cases}
	\end{equation} 
	We prove below, in Proposition~\ref{Lambdapower}, that for $\smallsimplex \under \bigsimplex\in D^l_k$ and any $m>l$ we have
	$$
	\Lambda^{m}(\smallsimplex\under\bigsimplex)= \begin{cases}
		\sum_{j=0}^{k-l} \bigsimplex_{j}\under \bigsimplex_{\leq k-l} &\mbox{ if }\smallsimplex=\bigsimplex_{\geq k-l}\\
		0 & \mbox{ otherwise.}
	\end{cases}
	$$
By definition $\kappa(\bigsimplex_{j}\under \bigsimplex_{\leq k-l})=\delta_{j,k-l}\bigsimplex_{\leq k-l}$, therefore $\dcap=\kappa\Lambda^m$ for $m > \dim(X)$.

	To show that $\dcap$ is a homotopy inverse to $\epsilon$, first note that for a $k$-simplex $\alpha$
	$$
	\dcap(\epsilon \bigsimplex)= \dcap\left(\sum_{i=0}^k \bigsimplex_i \under \bigsimplex \right)= \bigsimplex,
	$$
	so $\dcap \epsilon=\id$. The other composite gives, when evaluated on a generator $\smallsimplex\under \bigsimplex \in D^l_k$,
	\begin{align*}\label{epccomp}
		\epsilon\dcap(\smallsimplex\under\bigsimplex)&= \begin{cases}
			\epsilon(\bigsimplex_{\leq k-l})= \sum_{j=0}^{k-l} \bigsimplex_j \under \bigsimplex_{\leq k-l} & \mbox{ if }\smallsimplex=\bigsimplex_{\geq k-l}\\ 0 &\mbox{ otherwise}
		\end{cases}\\ & = \Lambda^m(\smallsimplex\under\bigsimplex),
	\end{align*}
	for $m>l$. The map $\Lambda$ is chain homotopic to the identity by construction, so any power of it is. This shows that $\epsilon\dcap$ is homotopic to the identity and we are done.
\end{proof}

This completes the proof of Lemma~\ref{cisnatural}, establishing that $\mathfrak{c}$ is natural, and finishes the proof of Theorem~\ref{standardexactrowspecseq}. In the proof of Lemma~\ref{cishomotopyinverse} we used:

\begin{proposition}[Python Proposition]\label{pythonlemma}
 	When $m > \dim(X)$ the chain map $\Lambda^m$ projects onto the kernel of $i^0$. 
\end{proposition}
We will prove this in two steps. The first uses: 
\begin{lemma}\label{Lambdaabstract}
	Let $\shcochain\in D^l_k$. Then for the map $\Lambda=\id - \lambda d^D - d^D\lambda$ we have
	$$
	\Lambda \shcochain= \begin{cases}
		(-1)^{k-l}\left(d \lambda \shcochain - \lambda d \shcochain\right)& \mbox{ for }l>0,\\
		\shcochain- \lambda i \shcochain & \mbox{ for }l=0.
	\end{cases}	
	$$
\end{lemma}

\begin{proof}
	The total differential on $D^l_k$ is $d^{D,l}_k=(-1)^{k-l}d_k+ i^l$, so we have
	$$
	\Lambda = \id -(-1)^{k-l}\lambda d -(-1)^{k-l-1}d \lambda-\lambda i - i\lambda.
	$$
	We showed that $\lambda i + i \lambda  =\id$ when $l>0$ in Lemma~\ref{commilambda}, so the above simplifies to the required result. On $D^0_\bullet$ the map $\lambda$ is trivial, and after removing trivial terms we attain $\Lambda =\id-i\lambda$.
\end{proof}

In particular this tells us that $\Lambda$ shifts elements along the diagonal. The second step in the proof of Proposition~\ref{pythonlemma} uses:

\begin{lemma}\label{Lambdaproj}
	Let $\shcochain\in D^0_{k}$, then $\Lambda^2\shcochain=\Lambda\shcochain$ and
	$$
	i \Lambda \shcochain=0.
	$$
\end{lemma}
\begin{proof}
	By Lemma~\ref{Lambdaabstract}, $\Lambda$ acts as $\id-\lambda i$ on $D^0_k$. In Lemma~\ref{commilambda} we showed that $\lambda i + i \lambda  =\id$ when $l>0$, so that
	$$
	(\lambda i)^2=\lambda (i \lambda) i = \lambda (\id -\lambda i) i  =\lambda i - \lambda i^2=\lambda i,
	$$
	as $i^2=0$. Hence $\Lambda$ on $D^0_k$ is complementary to an idempotent, hence also idempotent. For the second claim, note
	$$
	i \Lambda = i(\id - \lambda i)= i- i\lambda i=i-(\id-\lambda i)i=0,
	$$
	again using the fact that $\lambda i + i \lambda  =\id$ when $l>0$.
\end{proof}

\begin{proof}[Proof of Proposition~\ref{pythonlemma}]
	By Lemma~\ref{Lambdaabstract}, the map $\Lambda$ maps $D^l_k$ to $D^{l-1}_{k-1}$ for $l>0$, and $D^0_k$ to itself. On $D^0_k$, Lemma~\ref{Lambdaproj} tells us that $\Lambda$ projects onto $\ker i^0$. As $D^l_k$ is trivial when $l > \dim(X)$, the image of $\Lambda^{\dim(X)}$ is contained in $D^0_\bullet$. So for $m > \dim(X)$ the map $\Lambda^m$ projects onto $\ker i^0$. 
\end{proof}

\begin{remark}
	Note that if $\shcochain\in D^l_k$ is $d^D$ closed, then using Lemma~\ref{commilambda} we get
	$$
	\Lambda \shcochain = \shcochain - d^D \lambda \shcochain= \shcochain- i \lambda \shcochain-(-1)^{k-l+1} d \lambda \shcochain = (-1)^{k-l}d \lambda \shcochain, 
	$$
	so $\Lambda$ acts by lifting horizontally and subsequently taking the vertical differential on closed forms. Furthermore
	$$
	i \Lambda \shcochain= (-1)^{k-l} i d \lambda \shcochain = (-1)^{k-l} d i \lambda \shcochain = (-1)^{k-l} d \shcochain =0,
	$$
	so $\Lambda \shcochain \in \ker i$ and has a unique lift $\lambda\Lambda \shcochain \in D^{l-2}_{k-1}$. This corresponds to the connecting homomorphism in the usual snake lemma, our Python Proposition~\ref{pythonlemma} can be viewed as a form of the snake lemma with iterated lifts.
\end{remark}

In the proof of Theorem~\ref{standardexactrowspecseq} we also used the explicit value of powers of $\Lambda$ on generators. In computing this we streamline the notation by using  Kronecker delta symbols defined, for simplices $\smallsimplex$ and $\smallsimplexone$, as
$$
\delta_{\smallsimplex,\smallsimplexone}=\begin{cases}
	1 &\mbox{ when }\smallsimplex=\smallsimplexone\\ 0 & \mbox{ otherwise.}
\end{cases}
$$
The first step in computing $\Lambda$ on generators is:

\begin{lemma}\label{Lambdaongens}
	Let $\smallsimplex \under \bigsimplex \in D^l_k$. Then
	$$
	\Lambda( \smallsimplex\under \bigsimplex)= \begin{cases}(\delta_{\smallsimplex_l, \bigsimplex_k}- \delta_{\smallsimplex_l, \bigsimplex_{k-1}})\smallsimplex_{\leq l-1}\under\bigsimplex_{\leq k-1}& \mbox{ for }l>0\\
		\delta_{\smallsimplex,\bigsimplex_k} \sum_{j=0}^k \bigsimplex_j \under \bigsimplex & \mbox{ for }l=0.
	\end{cases}	
	$$
\end{lemma}

\begin{proof}
	By Lemma~\ref{Lambdaabstract}, for $l>0$ we need to compute $\lambda d $ and $d \lambda$ on $\smallsimplex \under \bigsimplex\in D^l_k$. Equation~\eqref{lambdaongen} gives:
	$$
	\lambda d \smallsimplex \under \bigsimplex = \sum_{j=0}^k(-1)^j \lambda(\smallsimplex \under\bigsimplex_{\langle j \rangle}) = \sum_{j=0}^{k}(-1)^{j+l}\delta_{\smallsimplex_l, (\bigsimplex_{\langle j \rangle})_{k-1}}\smallsimplex_{\leq l-1} \under\bigsimplex_{\langle j \rangle},
	$$
	Note that
	$$
	(\bigsimplex_{\langle j \rangle})_{k-1}= \begin{cases}
		\bigsimplex_{k}& \mbox{ for }j<k \\ \bigsimplex_{k-1} &\mbox{ for }j=k,
	\end{cases}
	$$
	so we have
	$$
	\lambda d \smallsimplex \under \bigsimplex=\begin{cases}
		\sum_{j=0}^{k-1}(-1)^{j+l}\smallsimplex_{\leq l-1} \under\bigsimplex_{\langle j \rangle} &\mbox{ for }\smallsimplex_l= \bigsimplex_{k}\\
		 (-1)^{k-l}\smallsimplex_{\leq l -1}\under \bigsimplex_{\leq k-1} &\mbox{ for }\smallsimplex_l= \bigsimplex_{k-1} \\
   0 &\mbox { otherwise.}
	\end{cases} 
	$$
	For $d\lambda$ we get:
	$$
	d \lambda \smallsimplex \under \bigsimplex = (-1)^{l} \delta_{\smallsimplex_l, \bigsimplex_k} d\smallsimplex_{\leq l-1}\under \bigsimplex= \delta_{\smallsimplex_l, \bigsimplex_k} d\smallsimplex_{\leq l-1}\under \bigsimplex=  \delta_{\smallsimplex_l, \bigsimplex_k}\sum_{j=0}^k (-1)^{j+l}\smallsimplex_{\leq l-1} \under\bigsimplex_{\langle j \rangle}.
	$$
	When $\smallsimplex_l=\bigsimplex_k$, all but the $j=k$ summands in $d\lambda $ and $\lambda d$ match. So we find
	$$
	(-1)^{k-l}\left(d \lambda \shcochain - \lambda d \shcochain\right)\smallsimplex\under \bigsimplex= (\delta_{\smallsimplex_l, \bigsimplex_k}- \delta_{\smallsimplex_l, \bigsimplex_{k-1}})\smallsimplex_{\leq l-1}\under \bigsimplex_{\leq k-1}.
	$$

	For $l=0$, Lemma~\ref{Lambdaabstract} tells us that we have $\Lambda= \id-\lambda i^0$. We found $\lambda i^0$ in Lemma~\ref{commilambda}, and using that calculation gives:
	$$
	\Lambda(\bigsimplex_m \under \bigsimplex) = \bigsimplex_m \under \bigsimplex-\lambda i^0(\bigsimplex_m \under \bigsimplex) = \begin{cases}
		\sum_{j=0}^{k}\bigsimplex_j \under \bigsimplex &\mbox{ for }m=k\\ 0 &\mbox{ otherwise.}
	\end{cases}
	$$
	This is what we wanted to show.
\end{proof}

We first prove a short lemma before showing that powers of $\Lambda$ behave nicely.

\begin{lemma}\label{niceidentity}
Suppose that $\sigma$ is an $n$-simplex and $\alpha$ is an $n+1$-simplex. If $\partial$ is the usual simplicial chain map then $$\sigma^*(\partial \alpha)=\sigma_{\leq{k-1}}^*(\partial\alpha_{\leq k})\sigma_{\geq k}^*(\partial \alpha_{\geq k}),$$ for $1 \leq k \leq n$.   
\end{lemma}

\begin{proof}
Note that \begin{equation}
	\smallsimplex^*(\partial \bigsimplex)=\begin{cases}
		(-1)^j &\mbox{ if }\smallsimplex=\bigsimplex_{\langle j \rangle}\\ 0& \mbox{ otherwise.}
	\end{cases}	
	\end{equation}
If both $\sigma_{\leq{k-1}}^*(\partial\alpha_{\leq k})$ and $\sigma_{\geq k}^*(\partial \alpha_{\geq k})$ are nonzero, then $\sigma_{\leq{k-1}}$ and $\sigma_{\geq k}$ are faces of $\alpha$, so $\sigma < \alpha$ and $\sigma=\alpha_{\langle j \rangle}$ for some $j$. If $j \leq k$ then $\sigma_{\leq k-1}=(\alpha_{\leq k})_{\langle j \rangle}$ and $\sigma_{\geq k} =(\alpha_{\geq k})_{\langle 0 \rangle}$, therefore \[\sigma_{\leq{k-1}}^*(\partial\alpha_{\leq k})\sigma_{\geq k}^*(\partial \alpha_{\geq k})=(-1)^j(-1)^0=(-1)^j.\] If $j \geq k$ then $\sigma_{\leq {k-1}}=(\alpha_{\leq k})_{\langle k \rangle}$ and $\sigma_{\geq k}=(\alpha_{\geq k})_{\langle j-k \rangle}$, therefore
\[\sigma_{\leq{k-1}}^*(\partial\alpha_{\leq k})\sigma_{\geq k}^*(\partial \alpha_{\geq k})=(-1)^k(-1)^{j-k}=(-1)^j.\qedhere\]
\end{proof}

\begin{lemma}\label{Lambdalowpower}
	Let $\smallsimplex \under \bigsimplex\in D^l_k$ and $1\leq m\leq l$. Then
	$$
	\Lambda^m(\smallsimplex \under \bigsimplex)= \smallsimplex_{\geq l-m+1}^*( \partial \bigsimplex_{\geq k-m}) \smallsimplex_{\leq l-m}\under \bigsimplex_{k-m}.
	$$
	Here $\partial$ denotes the usual chain differential.
\end{lemma}
\begin{proof}
	We will use induction on $m$. The base step is Lemma~\ref{Lambdaongens}, observing that
	$$
	\smallsimplex_{l}^* (\partial\bigsimplex_{\geq k-1})=\delta_{\smallsimplex_l, \bigsimplex_k}- \delta_{\smallsimplex_l, \bigsimplex_{k-1}}.
	$$
	For the induction step, assume we have established the formula for $m<l$. Then
	\begin{align*}
	\Lambda^{m+1}(\smallsimplex\under \bigsimplex)&= \smallsimplex_{\geq l-m+1}^*( \partial\bigsimplex_{\geq k-m}) \Lambda(\smallsimplex_{\leq l-m}\under \bigsimplex_{\leq k-m})\\
	&= \smallsimplex_{l-m}^* (\partial(\bigsimplex_{k-m-1}\star \bigsimplex_{k-m})) \smallsimplex_{\geq l-m+1}^*(\partial \bigsimplex_{\geq k-m}) \smallsimplex_{\leq l-m-1}\under \bigsimplex_{\leq k-m-1}.
	\end{align*}
	Now apply the $k=1$ version of Lemma~\ref{niceidentity} with $\sigma_{\geq l-m}$ and $\alpha_{\geq k-m-1}$.
\end{proof}

Higher powers of $\Lambda$ give:
\begin{proposition}\label{Lambdapower}
	For $\smallsimplex\under \bigsimplex \in D^l_k$ we have
	$$
	\Lambda^{m}(\smallsimplex\under\bigsimplex)= \begin{cases}
		\sum_{j=0}^{k} \bigsimplex_{j}\under \bigsimplex_{\leq k-l} &\mbox{ if }\smallsimplex=\bigsimplex_{\geq k-l}\\
		0 & \mbox{ otherwise.}
	\end{cases}
	$$
	for all $m\geq l+1$. 
\end{proposition}
\begin{proof}
	We just need to apply $\Lambda$ to the expression for $\Lambda^l$ from Lemma~\ref{Lambdalowpower}, using Lemma~\ref{Lambdaongens}. This gives
	$$
	\Lambda^{l+1}(\smallsimplex\under \bigsimplex)= \smallsimplex_{\geq 1}^*( \partial\bigsimplex_{\geq k-l}) \Lambda(\smallsimplex_{0}\under \bigsimplex_{\leq k-l})=\smallsimplex_{\geq 1}^*( \partial\bigsimplex_{\geq k-l})\smallsimplex_0^*(\bigsimplex_{k-l}) \sum_{j=0}^{k-l}\bigsimplex_{j}\under \bigsimplex_{\leq k-l}.
	$$
	For this to be non-zero we need $\smallsimplex_0=\bigsimplex_{k-l}$. It follows that $\smallsimplex_{\geq 1}^*(\partial \bigsimplex_{\geq k-l})$ is nonzero only if  $\smallsimplex_{\geq 1}=\bigsimplex_{\geq k-l+1}$. These together give $\smallsimplex=\bigsimplex_{\geq k-l}$, and we are done.
\end{proof}

\subsubsection{The product Mayer--Vietoris complex}\label{prodmvcompsec}
As the arguments used in the proofs above are chain-by-chain, they also apply to the double complex $D^{\sqcap}$ which is defined analogously to $D$ with products instead of direct sums. This gives:
\begin{corollary}\label{locallyfinitecor}
	The morphism 
	$$
	\mathfrak{c}^{\sqcap}\colon (D^{\sqcap})^\bullet_\bullet \rar C_\bullet^\textnormal{lf}(X,L^\textnormal{vc})
	$$
	defined by the same formula as $\mathfrak{c}$ from Theorem~\ref{standardexactrowspecseq} induces an isomorphism
	$$
	H_*(\mathbf{Tot}D^{\sqcap}) \cong H_*^\textnormal{lf}(X,L^\textnormal{vc}).
	$$
\end{corollary}

In particular, $\mathbf{Col}D$ and $\mathbf{Col}D^\sqcap$ abut to $H_{*}(X,L^\textnormal{vc})$ and $H_{*}^\textnormal{lf}(X,L^\textnormal{vc})$ respectively, and in the Cohen--Macaulay case this gives an isomorphism between the relative homology of $X$ and $L^\textnormal{vc}$ and the cohomology of the local homology sheaf over $L$. Comparing Equations~\eqref{cupversion1} and~\eqref{doublecomplexhomiso}, we see that this isomorphism is closely related to the cap product, and we will show below that it is in fact given by capping with a fundamental class.

\subsubsection{The Mayer--Vietoris complex with coefficients}\label{s:mvwithcoefs}
Applying $\G$-transport (Definition~\ref{d:gtransport}) to the MV-double complex we get
$$
D^l_k(L;\mathcal{G})=\bigoplus_{\smallsimplex \in L_l}C_k(X, X- \ostar \smallsimplex;\mathcal{G}),
$$
Applying $\mathcal{G}$-transport to  Theorem~\ref{standardexactrowspecseq} shows that $\mathfrak{c}_{\mathcal{G}}\colon \mathbf{Tot}D(L;\mathcal{G}) \rar C_*(X,L\vc;\mathcal{G})$ given on $ \smallsimplex\under \psi_\bigsimplex \in D^l_k(L;\mathcal{G})$ by 
$$
\mathfrak{c}_{\mathcal{G}}( \smallsimplex\under \psi_\bigsimplex)= \begin{cases}
	\mathcal{G}(\bigsimplex> \bigsimplex_{\leq k-l}) \psi_\bigsimplex &\mbox{ for }\smallsimplex=\bigsimplex_{\geq k-l}\\ 0 \mbox{ otherwise }
\end{cases}
$$
is a chain homotopy inverse to $\epsilon_\mathcal{G}$. A similar argument holds for the product MV complex.

\subsubsection{The dual Mayer--Vietoris complex}\label{dualsection}
We now establish the analogue of Theorem~\ref{standardexactrowspecseq} needed for the duality statements involving the second version of the cap product (Equation~\eqref{capnopairing}). To do this, we consider the dual double complex of $R$-bimodules
$$
B^\star_\bullet(L)=\hom_{\rmod}(D^\bullet_\star(L),R)
$$
obtained by taking left $R$-linear maps. Observe that this is equivalently the $\underline{R}$-dual (Definition~\ref{d:fdual}) of the MV double complex where $\underline{R}$ is the constant sheaf with stalks $R$. Note that 
$$
B^k_l=\prod_{\smallsimplex \in L_l} C^k(X,X-\ostar\smallsimplex ).
$$
Observe that the column spectral sequence has
$$
\mathbf{Col}B_{1,l}^k= H^k\left(\prod_{\smallsimplex \in L_l}C^k(X,X-\ostar \smallsimplex)\right)\cong \prod_{\smallsimplex \in L_l} h^k(\smallsimplex)
$$
as its $E_1$-page, and the $E_2$-page is therefore
$$
\mathbf{Col}B_{2,l}^k= H_l^{\lf}(L,h^k).
$$
We will show that this abuts to $H^*(X,L^\textnormal{vc})$, and establish the analogous statement for compactly supported cohomology and cosheaf homology. In doing this, we will also use the double complex
$$
(B_c(L))^k_l= \bigoplus_{\smallsimplex \in L_l}C^k(X,X-\ostar \smallsimplex),
$$
for which the $E^2$-page of the column spectral sequence is $H_l(L, h^k)$.

\begin{theorem}\label{othercoefctheorem}
	Let $L$ be a full subcomplex of a locally finite oriented simplicial complex $X$. Define the maps
	\begin{align*}
		\mathfrak{c}^\vee\colon& C^q(X, L^\textnormal{vc}) \rar \mathbf{Tot}B^q		\\
		\mathfrak{c}^\vee_c \colon& C^q_c(X, L^\textnormal{vc}) \rar (\mathbf{Tot}B_c)^q 
	\end{align*}
	on generators by sending for a $q$-simplex $\bigsimplexone$ with at least one vertex in $L$ the generator $\bigsimplexone^*\in C^q(X,L^\textnormal{vc})$ to
	\begin{equation}\label{cveedef}
		\mathfrak{c}^\vee(\bigsimplexone^*)=\mathfrak{c}^\vee_c(\bigsimplexone^*)=\sum_{k=q}^\infty\sum_{\bigsimplex\in L_k} \bigsimplexone^*(\bigsimplex_{\leq q}) \bigsimplex_{\geq q}\under \bigsimplex^*.
	\end{equation}
	These induce natural isomorphisms of $R$-bimodules
	\begin{align*}
		H^q(X, L^\textnormal{vc}) &\cong H^q\mathbf{Tot}B	\\
	 H^q_c(X, L^\textnormal{vc}) &\cong H^q\mathbf{Tot}B_c.
	\end{align*}
\end{theorem}

We will prove this in several steps. First, we establish that $\mathfrak{c}^\vee$ is dual to $\mathfrak{c}$. From this the statement about non-compactly supported cohomology follows, and hence also the statement for compactly supported cohomology when $L$ is a finite subcomplex. For infinite subcomplexes we complete the argument using a limit over finite subcomplexes.

\begin{lemma}\label{cduallemma}
	The dual $\hom(\dcap, R)$ of $\mathfrak{c}$ from Theorem~\ref{standardexactrowspecseq} is $\mathfrak{c}^\vee$ from Theorem~\ref{othercoefctheorem}.
\end{lemma}
\begin{proof}
	We just need to check this for a generator $\bigsimplexone^*\in C^q(X,L^\textnormal{vc})$ like in the statement of Theorem~\ref{othercoefctheorem}. Evaluating $\mathfrak{c}^\vee(\bigsimplexone^*)$ on $\smallsimplex \under \gamma \in D_l^{k}$ gives
	\begin{align*}
		\mathfrak{c}^\vee(\bigsimplexone^*)(\smallsimplex\under \gamma)	&=\sum_{k'=q}^\infty \sum_{\bigsimplex\in L_{k'}} \bigsimplexone^*(\bigsimplex_{\leq q }) (\bigsimplex_{\geq q}\under \bigsimplex^*)(\smallsimplex \under \gamma)\\
														&=\sum_{k'=q}^\infty\sum_{\bigsimplex\in L_{k'}} \bigsimplexone^*(\bigsimplex_{\leq q}) \delta_{\bigsimplex_{\geq q},\smallsimplex}\delta_{\bigsimplex, \gamma}\\
														&= \bigsimplexone^*(\gamma_{\leq q})\delta_{\gamma_{\geq q},\smallsimplex}.
	\end{align*}
	For $\delta_{\gamma_{\geq q},\smallsimplex}$ to be non-zero $\gamma_{\geq q}$ needs to be an $l$-simplex, meaning that $q=k-l$. The image of $\bigsimplexone$ under the dual of $\mathfrak{c}$ evaluated on the same element is
	\begin{align*}
		\bigsimplexone^*(\mathfrak{c}(\smallsimplex \under \gamma)) &=\bigsimplexone^*\left(\sum_{\bigsimplex \in X_{k}}\delta_{\smallsimplex, \bigsimplex_{\geq k-l}}\delta_{\bigsimplex,\gamma}\bigsimplex_{\leq k-l}\right)\\
													&=\bigsimplexone^{*}(\gamma_{\leq k-l})\delta_{\gamma_{\geq k-l},\smallsimplex},
	\end{align*}
	so the dual of $\mathfrak{c}$ and $\mathfrak{c}^\vee$ indeed agree.
\end{proof}

This immediately gives us:
\begin{corollary}\label{cveeisocor}
	The morphism $\mathfrak{c}^\vee$ induces an isomorphism in cohomology for all full subcomplexes of $X$.
\end{corollary}
\begin{proof}
	By Theorem~\ref{standardexactrowspecseq}, $\dcap$ is a chain homotopy equivalence, so its image $\dcap^\vee$ under the additive functor $\hom(-, R)$ is also a chain homotopy equivalence.
\end{proof}

This gives us all we need to prove Theorem~\ref{othercoefctheorem}.

\begin{proof}[Proof of Theorem~\ref{othercoefctheorem}]
	The statement for $\dcap^\vee$ is the content of Corollary~\ref{cveeisocor}. For $\dcap^\vee_c$, let $K$ be a full finite subcomplex. Then the double complexes $B(K)$ and $B_c(K)$ agree, so by Corollary~\ref{cveeisocor} the map $\mathfrak{c}_{c,K}$ induces an isomorphism between $H_c^l(X,K\vc)$ and $H^l\mathbf{Tot}B$. Observe that if $K'$ is another finite full subcomplex and $\iota\colon K\hookrightarrow K'$ we have by dualizing Lemma~\ref{cisnatural} the commutative diagram
	$$
	\begin{tikzcd}
		H^*(X, K\vc) \arrow[r,"\mathfrak{c}^\vee_K", "\cong"'] &H^*(B_c(K)) \\		
		H^*(X, (K')\vc) \arrow[r,"\mathfrak{c}^\vee_{K'}", "\cong"']\arrow[u, "\iota^*"] &H^*(B_c(K')).\arrow[u,"\iota^*"]
	\end{tikzcd}
	$$
	For a full subcomplex $L$ taking a colimit over full finite subcomplexes of $L$ gives
	$$
	\colim_{K\subset L}H^*(X, K\vc) \cong H^*_c(X, L\vc)
	$$
	and
	$$
	\colim_{K\subset L} B_c(K) \cong B_c(L).
	$$
	Commutativity of the diagram above means that the $\mathfrak{c}^\vee_K$ induce $\mathfrak{c}_c^\vee$ in the colimit to give
	$$
	\mathfrak{c}_c^\vee\colon  H^*_c(X, L\vc)\xrightarrow{\cong}H^*(B_c(L)),
	$$
	which is what we wanted to show.
\end{proof}

We conclude this section with the observation that the arguments above extend to cohomology with sheaf coefficients, using  the $\mathcal{F}$-dual and $\mathcal{G}$-transport functors (Definition~\ref{d:gtransport}):

\begin{theorem}\label{dual with coefs theorem}
    Let $L$ be a full subcomplex of a locally finite oriented simplicial complex $X$ and let $\mathcal{F}$ be a sheaf on $X$. Then the $\mathcal{F}$-transport of $\mathfrak{c}^\vee$ is the map 
    $$
    (\mathfrak{c}^\vee)_\mathcal{F}\colon C^\bullet(X, L^\textnormal{vc};\F) \to \mathbf{Tot}B(L)_{\mathcal{F}}^\bullet
    $$
    that is defined on generators by sending for a $q$-simplex $\bigsimplexone$ with at least one vertex in $L$ the generator $\bigsimplexone\under \psi_\bigsimplexone\in C^q(X,L^\textnormal{vc}; \mathcal{F})$ to
	\begin{equation}\label{eq: formula for c F dual}
		(\mathfrak{c}^\vee)_\mathcal{F}(\bigsimplexone\under \psi_\bigsimplexone)=\sum_{k=q}^\infty\sum_{\bigsimplex\in L_k} \bigsimplexone^*(\bigsimplex_{\leq q}) \bigsimplex_{\geq q}\under F(\bigsimplexone > \bigsimplex_{\geq q}) \psi_\bigsimplexone.
	\end{equation} 
    Furthermore, this map induces isomorphisms
    \begin{align*}
	 H^q(X, L^\textnormal{vc};\F)&\cong H^q\mathbf{Tot}B_\F	\\
	  H^q_c(X, L^\textnormal{vc};\F)&\cong H^q\mathbf{Tot}B_{c,\mathcal{F}}.
\end{align*}
\end{theorem}
\begin{proof}
    A direct application of Definition~\ref{d:gtransport} gives the formula \eqref{eq: formula for c F dual} for $(\mathfrak{c}^\vee)_\mathcal{F}$. By functoriality of $\mathcal{F}$-dualising, the $\mathcal{F}$-dual $
\mathfrak{c}^\mathcal{F}\colon C^\bullet(X, L\vc; \mathcal{F})\to \mathbf{Tot}B_{\mathcal{F}}^\bullet
$ for $\dcap$ induces an isomorphism in cohomology. 

Taking the $\mathcal{F}$-dual of Lemma~\ref{cduallemma} we see that $\mathfrak{c}^\mathcal{F}$ agrees with $(\mathfrak{c}^\vee)_\mathcal{F}$. Repeating the proof of Theorem~\ref{othercoefctheorem} with $\dcap^\vee$ replaced by $\dcap^\F$ establishes the isomorphism for compactly supported cohomology. 
\end{proof}

\subsection{The CM Duality Theorem}\label{dualitysection}
In this section we prove the CM Duality Theorem (Theorem~\ref{dualitytheorem}). We first define a fundamental class $\fund{X}\in H^\lf_n(X;h^*)$ of an $n$-dimensional locally finite oriented simplicial complex $X$ and compute the cap product with this class (Lemma~\ref{capwithfund}). We compare this with the quasi-isomorphism $\dcap$ (from Theorem~\ref{standardexactrowspecseq}) and its dual $\dcap^\vee$ (from Theorem~\ref{othercoefctheorem}) in \textsection~\ref{fundclassvsdualsect}. When $X$ is locally CM, the CM Duality Theorem follows quickly in \textsection\ref{cmdualitysection}.

\subsubsection{The fundamental class}\label{fundclass} 
If $X$ is a $n$-dimensional complex, we have that $h^n(\smallsimplex)=\langle \smallsimplex^* \rangle$ whenever $\smallsimplex$ is an $n$-simplex. This allows us to define:
\begin{definition}
	The \emph{fundamental class} $[\fund{X}]\in H^{\textnormal{lf}}_n(X,h^*)$ of an $n$-dimensional complex $X$ is the class of the ``diagonal'' locally finite cycle with coefficients in the local cohomology cosheaf given by
	\begin{equation}\label{fundamentclass}
		\fund{X}=\sum_{\bigsimplex \in X_n}\bigsimplex\under\bigsimplex^* \in \prod_{\bigsimplex \in X_n} h^*(\bigsimplex).
	\end{equation}
\end{definition}

To see that this chain is closed observe that at $\smallsimplex \in X_{n-1}$
$$
d\fund{X}_\smallsimplex = \sum_{\bigsimplex > \smallsimplex}[\smallsimplex: \bigsimplex]\smallsimplex\under \bigsimplex^{*}= \delta(\smallsimplex\under \smallsimplex^*) \in \textnormal{Im}(C^{n-1}(X, X-\ostar \smallsimplex) \xrightarrow{\delta}  C^{n}(X, X-\ostar \smallsimplex)),
$$
where $[\smallsimplex:\bigsimplex]=\pm 1$ tracks whether $\smallsimplex$ is an odd or even face of $\bigsimplex$.  

Capping with the fundamental class behaves quite nicely:

\begin{lemma}\label{capwithfund}
	Let $L$ be a subcomplex of an $n$-dimensional locally finite oriented simplicial complex $X$ with $L\vc$ before $L$ (Definition~\ref{vcandbeforedef}). Let $\shcochain \in C_c^l(L,h_n)$ and write $\shcochain=\sum \shcochain_{ \bigsimplex}^\smallsimplex \smallsimplex \under \bigsimplex$. Taking the relative cap product from Proposition~\ref{usefulrelativecups} with the fundamental class gives
	\begin{equation}\label{capwithfundeq1}
		\fund{X}\cap \shcochain= \sum_{\bigsimplex \in X_n}\shcochain^{\bigsimplex_{\geq n-l}}_{\bigsimplex}\bigsimplex_{\leq n-l}\in C_{n-l}(X,L\vc).
	\end{equation}
	Similarly, for $\intcochain \in C^l(X,L)$ written as $\intcochain=\sum \intcochain_{ \bigsimplex}^\smallsimplex \smallsimplex \under \bigsimplex$ we have
	\begin{equation}\label{capwithfundeq2}
		\fund{X}\cap \intcochain= \sum_{\bigsimplex \in X_n}\intcochain^{\bigsimplex_{\geq n-l}}_{\bigsimplex}\bigsimplex_{\leq n-l}\under \bigsimplex^*\in C_{n-l}(L\vc;h^n).
	\end{equation}
\end{lemma}
\begin{proof}
Capping with the representative of the fundamental class gives
\begin{align}
	\fund{X}\cap \shcochain&= \sum_{\smallsimplex \in L_l}\sum_{\bigsimplex\geq \smallsimplex} \shcochain_{ \bigsimplex}^\smallsimplex \fund{X}\cap (\smallsimplex \under \bigsimplex)\notag\\
	&=\sum_{\smallsimplex \in L_l}\sum_{\bigsimplex\geq \smallsimplex}\sum_{\bigsimplexone\in X_n} \shcochain_{ \bigsimplex}^\smallsimplex (\bigsimplexone\under \bigsimplexone^*) \cap (\smallsimplex \under \bigsimplex),\label{fundinter}	
\end{align}
using Equation~\eqref{fundamentclass}. The cap product between the generators is given in Equation~\eqref{cap1ongens}:
\begin{align*}
	(\bigsimplexone\under \bigsimplexone^*)\cap \smallsimplex\under \bigsimplex	&=\delta_{\bigsimplexone_{\geq n- l}, \smallsimplex} \langle \bigsimplexone^*,\bigsimplex \rangle \bigsimplexone_{\leq n-l}\\
	&= \delta_{\bigsimplexone_{\geq n-l}, \smallsimplex}\delta_{\bigsimplexone, \bigsimplex}\bigsimplexone_{\leq n-l}.
\end{align*}
Plugging this into Equation~\eqref{fundinter} now gives Equation~\eqref{capwithfundeq1}. Equation~\eqref{capwithfundeq2} is proved similarly. 
\end{proof}

\subsubsection{Capping with the fundamental class versus $\dcap$ and $\dcap^\vee$}\label{fundclassvsdualsect} We now compare capping with the fundamental class with the map $\dcap$ from Theorem~\ref{standardexactrowspecseq}. When $X$ is $n$-dimensional we have the maps $h_n(\smallsimplex)\hookrightarrow C_n(X,X-\ostar\smallsimplex)$ including the homology as the kernel of the differential. This gives a map (see Figure~\ref{mvdoublecomplexfigure} for a visual aid)
$$
C_c^l(L; h_n|_L) \xrightarrow{\iota} D^l_n(L)=\bigoplus_{\smallsimplex \in X_l} C_n(X, X-\ostar \smallsimplex).
$$
Similarly, we want to compare $\dcap^\vee$ (and its compactly supported version) from Theorem~\ref{othercoefctheorem} with $\fund{X}\cap -$. For $n$-dimensional $X$ we have a map induced from the quotient maps $C^n(X,X-\ostar\smallsimplex )\rar h^n(\smallsimplex)$ to give
$$
B^l_n(L\vc) \xrightarrow{q}  C^\lf_{n-l}(L\vc;h^n),
$$
that we will use for the comparison.

\begin{proposition}\label{capvsc}
	Let $X$ be an $n$-dimensional locally finite oriented simplicial complex with fundamental class $[\fund{X}]\in H^\lf_n(X;h^n)$ and subcomplex $L$, oriented so that $L\vc$ is before $L$. Then the composite
	$$
	C_c^l(L; h_n|_L) \xrightarrow{\iota} D^l_n(L) \xrightarrow{\dcap} C_{n-l}(X, L\vc)
	$$
	is equal to the cap product $\fund{X}\cap -$. The corresponding statement for non-compactly supported cohomology and locally finite homology also holds.
	Similarly 
	$$
	C_c^l(X,L; R) \xrightarrow{\dcap^\vee} B^l_n(L\vc) \xrightarrow{q}  C_{n-l}(L\vc;h^n)
	$$
	agrees with $\fund{X}\cap-$ and the analogous statement for non-compactly supported cohomology and locally finite homology holds.
\end{proposition}
\begin{proof}
	The first statement follows immediately by comparing the formula for $\fund{X}\cap -$ from Lemma~\ref{capwithfund} with the expression for $\dcap$ from Theorem~\ref{standardexactrowspecseq}. 
	
	For the second we need to take a little care. Theorem~\ref{othercoefctheorem} provides a map out of $C^q(X, L^\textnormal{vc})$, so we would like to apply this theorem with $L\vc$ playing the role of $L$. However, we have oriented $L\vc$ before $L$. This means that to apply Theorem~\ref{othercoefctheorem} we should reverse the orientation of $X$. The map $\dcap^\vee$ (see Equation~\eqref{cveedef}) then becomes the map that takes $\bigsimplexone^*\in C^l_c(X,L)$ to a chain in $B(L\vc)$ with components in $B^l_n(L\vc)$ given by
	$$
	\sum_{\bigsimplex\in L_n} \bigsimplexone^*(\bigsimplex_{\geq n-l}) \bigsimplex_{\leq n-l}\under \bigsimplex^*.
	$$
	Here we used that reversing the orientation of $X$ exchanges front faces and back faces. Composing this with the map $B^l_n(L\vc) \rar  C^\lf_{n-l}(L\vc;h^n)$ indeed agrees with the expression found in Lemma~\ref{capwithfund}.
\end{proof}

\subsubsection{The CM Duality Theorem}\label{cmdualitysection}
We have now gathered all we need to prove the CM Duality Theorem. There will be eight different duality isomorphisms in this theorem. These are subdivided into four that map out of cohomology with coefficients in $h_*$, and four that map out of $R$-cohomology. Within these groups there is a further subdivision into pairs based on whether one needs the local CM condition to hold on all of the simplicial complex or only at a subcomplex. In each of the resulting four pairs there is both a variant where the domain is the compactly supported cohomology or cohomology without restriction on support. The last part of the theorem concerns (co)homology with coefficients in (co)sheaves, we remind the reader of Definition~\ref{d: CM} of the notion of being CM over a (co)sheaf. 

\begin{theorem}[CM Duality]\label{dualitytheorem}
	Let $X$ be a locally finite complex of dimension $n$ with local homology sheaf $h_n$ and local cohomology sheaf $h^n$. Let $L$ be a full subcomplex of $X$ and pick an oriented of $X$ so that $L\vc$ is before $L$.
	Then the following maps
	\begin{enumerate}[label=\arabic*., ref=\arabic*]
		\item\label{covariantgroup}\begin{enumerate}[label=\alph*., ref=\arabic{enumi}.\alph*]
			\item when $X$ is locally CM at $L$ \label{poincarecptcosh} \begin{enumerate}[label=\roman*., ref=\arabic{enumi}.\alph{enumii}.\roman*]
					\item \label{poincare1ai} $H^l_c(L; h_n|_L) \xrightarrow{[\fund{X}]\cap -} H_{n-l}(X,L^\textnormal{vc};R)$,
					\item \label{poincare1aii}  $H^l(L; h_n|_L) \xrightarrow{[\fund{X}]\cap -} H_{n-l}^{\textnormal{lf}}(X,L^\textnormal{vc};R)$
				\end{enumerate}
				\item\label{poincarelocfincosh} when $X$ is locally CM
				\begin{enumerate}[label=\roman*., ref=\arabic{enumi}.\alph{enumii}.\roman*]
					\item \label{poincare2ai} $H^l_c(X,L;h_n) \xrightarrow{[\fund{X}]\cap -} H_{n-l}(L\vc;R)$
					\item \label{poincare2aii}  $H^l(X,L;h_n) \xrightarrow{[\fund{X}]\cap -} H^\textnormal{lf}_{n-l}(L\vc;R)$
			\end{enumerate}			
			\end{enumerate}
		\item \label{contravariantgroup} \begin{enumerate}[label=\alph*., ref=\arabic{enumi}.\alph*]
			\item when $X$ is locally CM at $L\vc$\label{poincarecptint} \begin{enumerate}[label=\roman*., ref=\arabic{enumi}.\alph{enumii}.\roman*]
				\item \label{poincare2bi}  $H^l_c(X,L;R) \xrightarrow{[\fund{X}]\cap -} H_{n-l}(L\vc;h^n|_{L\vc})$
					\item \label{poincare2bii}  $H^l(X,L;R) \xrightarrow{[\fund{X}]\cap -} H^\textnormal{lf}_{n-l}(L\vc;h^n|_{L\vc})$
			\end{enumerate}
		\item when $X$ is locally CM \label{pioncarelocfinint}\begin{enumerate}[label=\roman*., ref=\arabic{enumi}.\alph{enumii}.\roman*]
					\item \label{poincare1bi}  $H^l_c(L; R) \xrightarrow{[\fund{X}]\cap -} H_{n-l}(X,L^\textnormal{vc};h^n)$
			\item \label{poincare1bii} $H^l(L; R) \xrightarrow{[\fund{X}]\cap -} H^\textnormal{lf}_{n-l}(X,L^\textnormal{vc};h^n)$
		\end{enumerate}
		\end{enumerate}
	\end{enumerate}
	induced by taking the cap product with the fundamental class are all isomorphisms of $R$-bimodules.
	
	For a cosheaf $\mathcal{G}$, replacing ``locally CM'' by ``locally CM over $\mathcal{G}$'' in the hypotheses,  the maps from \ref{covariantgroup} with  $h_n$ replaced  by $(h_\G)_n$ and $R$ by $\G$ are isomorphisms. 

    Similarly, for a sheaf $\mathcal{F}$, replacing ``locally CM'' by ``locally CM over $\mathcal{F}$'' in the hypotheses,  the maps from \ref{contravariantgroup} with  $h^n$ replaced  by $(h^\F)_n$ and $R$ by $\F$ are isomorphisms. 
\end{theorem}

\begin{proof}
	\ref{poincare1ai}: Recall the MV double complex $D^\bullet_\bullet(L)$ from \textsection~\ref{mvdoublecompsect}. Because $X$ is CM at $L$ the column spectral sequence of this double complex collapses (Proposition~\ref{CMspecseqcollapse}) to give an isomorphism
	$$
	H^*_c(L;h_n|_L) \cong H_* \mathbf{Tot}D,
	$$
	induced by the inclusion of $C^*(L,h_n)$ as the $n$th row of the $E^1$-page. Combining this with Theorem~\ref{standardexactrowspecseq} gives an isomorphism
	$$
	H^l(L; h_n|_L) \xrightarrow{\cong} H_{n-l}(X,L^\textnormal{vc};R).
	$$
	By Proposition~\ref{capvsc} this isomorphism agrees with $[\fund{X}]\cap -$, establishing that the map in~\ref{poincare1ai} is an isomorphism.
	
	\ref{poincare1aii}: Observe that when $X$ is locally CM at $L$ we have
	$$
	H_*(\mathbf{Tot} D^\sqcap) \cong H^*(L; h^n|_L),
	$$
	by a similar spectral sequence argument as for $D$. Applying Corollary~\ref{locallyfinitecor} and Proposition~\ref{capvsc} then yields the result.
	
	\ref{poincare2bi} and~\ref{poincare2bii}: We prove the map from~\ref{poincare2bi} is an isomorphism, the argument for the map from~\ref{poincare2bii} is analogous. For $X$ locally CM at $L$, the double complex $B_c$'s column spectral sequence collapses to give
	$$
	H^l(\mathbf{Tot}B_c(L\vc)) \cong H_{n-l}(L\vc,h^n).
	$$
	Using Proposition~\ref{capvsc} now yields the result.
	
	The fact that the maps from \ref{poincarelocfincosh} and \ref{pioncarelocfinint} are isomorphisms is now a routine deduction
    . We present the argument here for \ref{poincare2ai}, the rest are similar. The short exact sequence for cochains (Equation~\eqref{usefulses}) restricts to a short exact sequence for cochains with compact support. This, together with the short exact sequence for chains and the (relative) cap product, combines into the diagram
	\begin{center}
		\begin{tikzcd}
			0 \arrow[r] 	&C^l_c(X,L;h_n)\arrow[r]\arrow[d,"\fund{X}\cap -"]	&C^l_c(X;h_n) \arrow[r]\arrow[d,"\fund{X}\cap -"]	&C^{l-1}_c(L;h_n|_L) \arrow[r]\arrow[d,"\fund{X}\cap -"]	& 0 \\
			0 \arrow[r] 	&C_{n-l}(L\vc)\arrow[r]								&C_{n-l}(X) \arrow[r]								&C_{n-l+1}(X,L\vc) \arrow[r]				&0.
		\end{tikzcd}
	\end{center}
	The proof of Proposition~\ref{usefulrelativecups} implies that this is a commutative diagram. Additionally, as $\fund{X}$ is closed $\fund{X}\cap -$ is a chain map, so we get the commutative diagram
	\begin{center}
		\begin{tikzcd}
			\dots \arrow[r] 	&H^{l-1}_c(L;h_n|_L) \arrow[r]\arrow[d,"\fund{X}\cap -","\cong"']	&H^l_c(X,L;h_n)\arrow[r]\arrow[d,"\fund{X}\cap -","\cong"']	&H^l_c(X;h_n) \arrow[r]\arrow[d,"\fund{X}\cap -","\cong"']	& \dots \\
			\dots \arrow[r] 	&H_{n-l+1}(X,L\vc) \arrow[r]								&H_{n-l}(L\vc)\arrow[r]								&H_{n-l}(X) \arrow[r]								&\dots
		\end{tikzcd}
	\end{center}
	for the long exact sequences. As $X$ is locally CM, it is also locally CM at $L$, so both the maps on the sides are isomorphisms by \ref{poincare1ai}. Applying the five-lemma then shows that the map in the middle is too.
	
	If $X$ is CM of dimension $n$ at $L$ over $\mathcal{G}$, then the column spectral sequence of $D(L;\mathcal{G})$ (see \textsection\ref{s:mvwithcoefs}) collapses to give
	$$
	H^*\mathbf{Tot}D(L;\mathcal{G}) \cong H^*_c(L;h^\mathcal{G}_*|_L).
	$$
	The $\G$-transport of Proposition~\ref{capvsc} gives that the map $C^n_c(L; h^\mathcal{G}_*|_L) \rar D^l_n(L;\mathcal{G})$ composed with $\dcap_\mathcal{G}$ agrees with $\fund{X}\cap_\G-$, where $\cap_\G$ is the $\G$-transport of $\cap$ introduced in \textsection\ref{s:capwithcoefs}. This allows us to adapt the arguments above (using the $\G$-transport of the short exact sequences from Equation~\eqref{usefulses} to see that the maps from \ref{covariantgroup} are isomorphisms with $h_n$ replaced by $(h_\G)_n$ and $R$ by $\G$.

    If $X$ is locally CM at $L$ over a sheaf $\mathcal{F}$, the column spectral sequence of the $\F$-transported double complex $B_{c,\mathcal{F}}(L\vc)$ collapses. The $\F$-transport of Proposition~\ref{capvsc} (using Theorem~\ref{dual with coefs theorem}) then gives that 
    $$
    H^l_c(X,L;\F) \xrightarrow{[\fund{X}]\cap -} H_{n-l}(L\vc;(h^\F)^n|_{L\vc})
    $$
    is an isomorphism. The fact that the rest of the maps from \ref{contravariantgroup} are isomorphisms with $h^n$ replaced by $(h^\F)^n$ and $R$ by $\F$ now follows by similar arguments to the ones presented above. 
\end{proof}

\begin{remark}
	For a triangulated manifold without boundary, the local homology sheaf over $\Z$ is a locally constant sheaf with stalks $\Z$. (Here locally constant means that all restriction maps are isomorphisms.) An orientation is the same as an isomorphism of sheaves of graded abelian groups between $h_*$ and the constant sheaf $\Z$ in degree $n$. This is equivalently (by the Universal Coefficient Theorem) an isomorphism of cosheaves of graded abelian groups between the constant cosheaf $\Z$ in degree $n$ and $h^*$. With these observations in mind one recovers the usual Poincar\'e duality for both oriented and unorientable manifolds from the Duality Theorem above. Note that for a non-orientable $n$-manifold $M$ there are two distinct versions:
	$$
	H^k(M;h_*) \cong H_{n-k}(M;\Z)
	$$
	as well as
	$$
	H^k(M;\Z) \cong H_{n-k}(M;h^*).
	$$
	When $M$ is orientable and has boundary $\partial M$, the local homology is trivial on the boundary and one sees that
	$$
	H^*(M;h_*) \cong H^*(M,\partial M;\Z),
	$$
	so the Duality Theorem gives Lefschetz duality.
\end{remark}

\section{Functoriality and Naturality}\label{functorialitysection}
This section is devoted to explaining in what sense the duality isomorphisms are natural, giving the relevant induced maps for the cohomology of the local homology sheaf and the homology of the local cohomology cosheaf in Section~\ref{sheafmaps}. This allows us to establish naturality of CM duality in Section~\ref{natdualsection}.

\subsection{Functoriality}\label{sheafmaps}
We introduce the class of functions for which we have functoriality in \textsection\ref{simplcomplmapssect}, and give their induced maps on (co)homology in \textsection\ref{inducedmapsect}.

\subsubsection{Star-local homeomorphisms}\label{simplcomplmapssect}
From simplicial (co)homology we have:

\begin{definition}
	Let $f\colon X \rar Y$ be a simplicial map between oriented simplicial complexes and assume that $f\smallsimplex$ has the same dimension as $\smallsimplex$. The \emph{orientation index $\ind{f}{\smallsimplex}$ of $f\smallsimplex$} is the sign of the permutation taking the vertices of $f\smallsimplex$ ordered according to $X$ to the vertices ordered using the ordering of $Y$. When $\ind{f}{\smallsimplex}=1$ for all $\smallsimplex\in X$, we say that $f$ is \emph{orientation preserving}.
\end{definition}

The orientation index satisfies $(-1)^i \ind{f}{\smallsimplex_{\langle j \rangle}}= (-1)^j \ind{f}{\smallsimplex}$, using this one proves:

\begin{lemma}\label{orientationindexfixesallourproblems}
		A map between oriented complexes $f: X\rar Y$ induces a chain map
	\begin{align*}
		f_*\colon C_l (X)& \rar C_l (Y)\\
		\smallsimplex	&\mapsto \ind{f}{\smallsimplex} f\smallsimplex,
	\end{align*}
	between the simplicial chain complexes of oriented simplices of $X$ and $Y$. 
\end{lemma}

For the local (co)homology (co)sheaf there is only a limited class of maps for which we have a good functoriality result, namely:

\begin{definition}\label{starlocalhomeodef}
	A \emph{star-local homeomorphism} $f\colon X \rar Y$ between simplicial complexes is a simplicial map such that for every $\smallsimplex \in Y$ and $\smallsimplexone\in f^{-1}\smallsimplex$ the map
	$$
	f|_{\st\smallsimplexone}\colon \st\smallsimplexone \rar \st \smallsimplex
	$$
	is a simplicial isomorphism. Such an $f$ between oriented complexes is \emph{orientation preserving} if additionally each $f|_{\st \smallsimplexone}$ preserves the orientation of $\st \smallsimplexone$. 
\end{definition}
One checks that if $f$ is star-local and $\smallsimplexone, \smallsimplextwo \in f^{-1}\smallsimplex$ are distinct, then $\st \smallsimplexone \cap \st \smallsimplextwo=\emptyset$.

\begin{remark}
One can show that any simplicial map that induces a local homeomorphism on geometric realisations is a star-local homeomorphism. 
\end{remark}

\subsubsection{The induced maps}\label{inducedmapsect}
We have the following maps induced by star-local homeomorphisms, we leave checking that these formulas define chain maps to the reader.

\begin{definition}\label{covariancelemma}
	Let $f\colon X\rar Y$ be a star-local homeomorphism. Then we have chain maps defined by
	\begin{alignat*}{4}
		f_! \colon	 C_c^*(&X; h^X_*)&\rar &C^*(Y, h^Y_*)\\
		 &\smallsimplex \under \bigsimplex &\mapsto &\ind{f}{\smallsimplex}\ind{f}{\bigsimplex}f\smallsimplex \under f\bigsimplex,
	\end{alignat*}
	and by
	\begin{alignat*}{4}
		f^! \colon	 C^\lf_*(&Y; h^*_Y) &\rar& C^\lf_*(X;h_X^*)\\
		 &\smallsimplex\under \bigsimplex^* &\mapsto &\sum_{\smallsimplexone \in f^{-1}\smallsimplex}\ind{f}{\smallsimplexone} \ind{f}{f|_{\st \smallsimplexone}^{-1}\bigsimplex}\smallsimplexone \under (f|_{\st \smallsimplexone}^{-1} \bigsimplex)^*.
	\end{alignat*}
	If $f$ has finite fibres, then $f$ also induces maps $f_! \colon C_c^*(X; h^X_*)\rar C_c^*(Y, h^Y_*)$ and  $f_! \colon C_c^*(X; h^X_*)\rar C_c^*(Y, h^Y_*)$, as well as $f^! \colon	 C_*(Y; h^*_Y) \rar C_*(X;h_X^*)$.
\end{definition}

\subsection{Naturality of duality}\label{natdualsection}
We now establish naturality of the duality isomorphisms. We first show that the cap product is natural with respect to star local homeomorphisms in \textsection\ref{natofcapsect}. In \textsection\ref{functfundclass} we show that the fundamental class is preserved under these maps, which allows us to conclude naturality for the duality isomorphisms in \textsection\ref{natdualizosect}.

\subsubsection{Naturality of the cap product}\label{natofcapsect}
The two versions of the cap product both satisfy a kind of naturality using the maps $f_!$ and $f^!$ from Corollary~\ref{covariancelemma}:
\begin{lemma}\label{naturalitylemma}
	Let $f\colon X\rar Y$ be an orientation preserving star-local homeomorphism of simplicial complexes. If $\shchain \in C^\lf_k(Y;h^*_Y)$ and $\shcochain \in C_c^l(X;h_*^X)$, then the cap product from Equation~\eqref{cupversion1} satisfies
	$$
	f_*(f^!\shchain \cap \shcochain) = \shchain \cap f_!\shcochain\in C_{k-l}(Y).
	$$
	The cap product from Equation~\eqref{cupversion2} satisfies, for $\shchain \in C^\lf_k(Y;h^*_Y)$ and $\intcochain \in C^l(Y)$,
	$$
	f^!\shchain\cap f^*\intcochain = f^!(\shchain \cap \intcochain)\in C_{k-l}(X;h^*_X).
	$$
	If $f$ is not orientation preserving, the above holds after passing to homology.
\end{lemma}

\begin{proof}
	The statement for when $f$ is not orientation preserving follows from Theorem~\ref{orderderindependencetheorem} and the rest of the content of this lemma. So assume that $f$ is orientation preserving. It suffices to prove the above statements on generators at the chain level. For the first equation let $\shchain= \smallsimplex\under \bigsimplex^*$ and $\shcochain=\smallsimplexone \under \bigsimplexone$. Then
	\begin{align*}
			f_*(f^! \shchain \cap \shcochain)&= f_*\left(\sum_{ \smallsimplextwo\in f^{-1} \smallsimplex} \smallsimplextwo \under (f|_{\st \smallsimplextwo}^{-1}\bigsimplex)^* \cap \smallsimplexone \under \bigsimplexone\right)& \textnormal{ (Lemma~\ref{covariancelemma})}\\
			&= f_*\left(\sum_{f\smallsimplextwo=\smallsimplex}\delta_{\smallsimplextwo_{\geq k-l}, \smallsimplexone} \langle(f|_{\st \smallsimplextwo}^{-1}\bigsimplex)^*, \bigsimplexone \rangle \smallsimplextwo_{\leq k-l} \right)&\textnormal{ (Equation~\eqref{cap1ongens})}\\
			& = \bigsimplex^*(f\bigsimplexone)\sum_{f\smallsimplextwo=\smallsimplex}\delta_{\smallsimplextwo_{\geq k-l}, \smallsimplexone}f(\smallsimplextwo_{\leq k-l}).& \textnormal{ (as }\langle(f|_{\st \smallsimplextwo}^{-1}\bigsimplex)^*, \bigsimplexone \rangle = \bigsimplex^*(f\bigsimplexone)).
	\end{align*}
	Note that the indices $\ind{f}{\smallsimplextwo}$ play no role as $f$ is orientation preserving. As $f$ is a star-local homeomorphism and the preimages of $\smallsimplex$ are disjoint, among those $\smallsimplextwo$ that $f$ maps to $\smallsimplex$ there can be at most one that has $\smallsimplexone$ as its back face. Furthermore, as $f$ is orientation preserving we have that $f(\smallsimplextwo_{\geq k-l})=f(\smallsimplextwo)_{\geq k-l}=\smallsimplex_{\geq k-l}$. This means
	$$
	f_*(f^! \shchain \cap \shcochain)= \bigsimplex^*(f\bigsimplexone)\delta_{\smallsimplex_{\geq k-l}, f\smallsimplexone}\smallsimplex_{\leq k-l}.
	$$
	On the other side we have (using Lemma~\ref{covariancelemma})
	$$
	\shchain \cap f_!\shcochain= (\smallsimplex\under \bigsimplex^*)\cap (f \smallsimplexone\under f \bigsimplexone)= \bigsimplex^*(f\bigsimplexone) \delta_{\smallsimplex_{\geq k- l},f \smallsimplexone}\smallsimplex_{\leq k-l},
	$$
	establishing the result. The second equality follows from a similar computation.
\end{proof}

\subsubsection{Functoriality of the fundamental class}\label{functfundclass}
The fundamental class is invariant under star-local homeomorphisms in the following sense:
\begin{lemma}\label{fundclassnatlemma}
	Let $f\colon X \rar Y$ be a star-local homeomorphism between oriented locally finite locally CM complexes of dimension n and denote by $h_X^n$ and $h^n_Y$ the local cohomology cosheaves of $X$ and $Y$, respectively. Then the map
	$$
	f^!\colon H^\lf_k(Y;h^n_Y) \rar H^\lf_k(X;h^n_X)
	$$
	takes the fundamental class of $Y$ to the fundamental class of $X$.
\end{lemma}

\begin{proof}
	Computing the image of the diagonal representative of the fundamental class of $Y$ under $f^!$ gives:
	$$
	f^{!}(\fund{Y})=f^!\left(\sum_{\bigsimplex\in Y_n}\bigsimplex\under \bigsimplex^*\right)=\sum_{\bigsimplex\in Y_n}\sum_{\bigsimplexone\in f^{-1}\bigsimplex}\ind{f}{\bigsimplexone}\ind{f}{f|_{\st \bigsimplexone}^{-1}\bigsimplex}\bigsimplexone\under (f|_{\st \bigsimplexone}^{-1} \bigsimplex)^*.
	$$
	We note that $f|_{\st \bigsimplexone}^{-1}\bigsimplex=\bigsimplexone$.
	The sets $\{f^{-1}\alpha\}_{\alpha \in Y_n}$ partition $X_n$, so we have
	$$
	f^!(\fund{Y})= 	\sum_{\bigsimplexone \in X_n} \bigsimplexone \under\bigsimplexone^* =\fund{X},
	$$
	using that $\ind{f}{\bigsimplexone}^2=1$.
\end{proof}

\subsubsection{Naturality of the duality isomorphisms}\label{natdualizosect}
The isomorphisms in the Duality Theorem~\ref{dualitytheorem} satisfy a particular kind of naturality for star-local homeomorphisms, covariant for the first four (under~\ref{covariantgroup}) and contravariant for the second four (under~\ref{contravariantgroup}).

\begin{theorem}\label{naturalityofduality}
	Let $f\colon (X,L) \rar (Y,K)$ be a star-local homeomorphism between pairs of locally Cohen--Macaulay complexes and full subcomplexes. Write $h^X_*$ (and $h_X^*$), and $h_*^Y$ (and $h^*_Y$) for the local (co)homology of $X$ and $Y$, respectively. Then for the duality isomorphism from Theorem~\ref{dualitytheorem}.\ref{poincare1aii} the diagram
	$$
	\begin{tikzcd}
		H_c^{l}(L;h^X_*)\arrow[r,"{[\fund{X}]\cap-}","\cong"']  \arrow[d,"f_!"]	& H_{n-l}(X,L\vc;R)\arrow[d,"f_*"]\\
		H^{l}(K;h^Y_*) \arrow[r,"{[\fund{Y}]\cap-}","\cong"'] 				& H^\lf_{n-l}(Y,K\vc;R)
	\end{tikzcd}
	$$
	where $f_!$ is the map introduced in Corollary~\ref{covariancelemma}, commutes. A similar diagram commutes for the isomorphism in Theorem~\ref{dualitytheorem}.\ref{poincare1bii}. If $f$ additionally has finite fibres, the corresponding diagrams for the isomorphisms from Theorem~\ref{dualitytheorem}.\ref{poincare1ai} and \ref{dualitytheorem}.\ref{poincare1bi} also commute. For the isomorphisms under Theorem~\ref{dualitytheorem}.\ref{contravariantgroup} the diagram
	$$
	\begin{tikzcd}
		H^{l}(L;R)\arrow[r,"{[\fund{X}]\cap-}",,"\cong"']					& H^\lf_{n-l}(X,L\vc;h_X^*)\\
		H^{l}(K;R) \arrow[r,"{[\fund{Y}]\cap-}",,"\cong"'] \arrow[u,"f^*"]	& H^\lf_{n-l}(Y,K\vc;h_Y^*)\arrow[u,"f^!"]
	\end{tikzcd}
	$$
	and its variants (subject to the condition that $f$ be proper for those involving compactly supported cohomology) for the other items commute.
\end{theorem}
\begin{proof}
	We know from Lemma~\ref{fundclassnatlemma} that $f^!$ takes the fundamental class of $Y$ to the one of $X$. So for the first diagram we are trying to show that for any cohomology class represented by a cocycle $\shcochain$ we have
	$$
	f_*([f^!\fund{Y}]\cap [\shcochain])=[\fund{Y}\cap f_!\shcochain],
	$$
	but this this the first part of Lemma~\ref{naturalitylemma}. Similarly, for the second diagram we are checking that
	$$
	f^!\fund{Y}\cap f^* \intcochain = f^!(\fund{Y}\cap \intcochain),
	$$
	this is the second part of Lemma~\ref{naturalitylemma}.
\end{proof}

\begin{remark}
	Let $M$ be an orientable $n$-manifold with an orientation reversing $\Z/2\Z$-action along $f\colon M \rar M$. Working with $R=\Z$, Lemma~\ref{fundclassnatlemma} tells us that the $\Z/2\Z$-action given by $f^!$ on $H_n^\lf(M;h^*)$ is trivial. However, $\Z/2\Z$ will act by $-1$ on $H^\lf_n(M;\Z)$, meaning that while $H_n^\lf(M;h^*)\cong H^\lf_n(M;\Z)$ as $\Z$-modules, they are not $\Z/2\Z$-equivariantly isomorphic. This shows that the CM fundamental class defined here is a different object from the usual fundamental class for orientable manifolds. In particular, the usual Poincar\'e duality for orientable manifolds is not natural with respect to orientation reversing homeomorphisms. 
\end{remark}

\appendix

\section{Orientation Independence of the Cap Product}\label{orindepapp}
We relegated a part of the proof that the cap product does not depend on the orientation of the simplicial complex to here. That the cap product is independent of the orientation is already non-trivial for the cap products with $R$-coefficients
\begin{align}\begin{split}\label{Rcoefcap}
	\cap \colon C^\lf_k(X;R)\otimes C^l(X;R) &\rar C^\lf_{k-l}(X)	\\
	\smallsimplex\otimes \smallsimplexone^* &\mapsto \smallsimplexone^*(\smallsimplex_{\geq k-l}) \smallsimplex_{\leq k-l},
\end{split}
\end{align}
and 
\begin{align}\begin{split}\label{Rcoefcapv2}
		\cap \colon C_k(X;R)\otimes C_c^l(X;R) &\rar C_{k-l}(X)	\\
		\smallsimplex\otimes \smallsimplexone^* &\mapsto \smallsimplexone^*(\smallsimplex_{\geq k-l}) \smallsimplex_{\leq k-l}.
	\end{split}
\end{align}
We will prove Proposition~\ref{orindepprop} alongside this, the difference between the proofs is cosmetic. In fact, for ease of notation we will only explicitly prove that the $R$-coefficient cap product is orientation independent and indicate what needs changing to establish Proposition~\ref{orindepprop}.

\subsection{Proof of Proposition \ref{orindepprop}}
In the proof below, we will use that one way to think about orientation of finite complexes is through a \emph{total ordering} of the vertices: this induces consistent orientations for all simplices of the complex by restriction. We first prove the following version of Proposition \ref{orindepprop} for the $R$-coefficient cap product from Equations~\eqref{Rcoefcap}~and~\eqref{Rcoefcapv2}:
\begin{proposition}\label{Rcoefordindep}
	Let $X$ be an oriented simplicial complex and $[u,w]$ be an oriented 1-simplex. Write $\tilde{C}_\bullet(X)$ for the chain complex associated to the orientation where $u$ and $w$ are swapped and $\tilde{\cap}$ for the associated cap product. Then the diagram
	\begin{center}
		\begin{tikzcd}
			C^\lf_k(X)\otimes C^l(X) \arrow[r,"\cap"] \arrow[d,"\cong"] & C^\lf_{k-l}(X)\arrow[d, "\cong"]\\
			\tilde{C}^\lf_k(X)\otimes\tilde{C}^l(X) \arrow[r,"\tilde{\cap}"]& \tilde{C}^\lf_{k-l}(X)
		\end{tikzcd}
	\end{center}
	(where the vertical maps are the isomorphisms from Lemma~\ref{homorderinv}) commutes up to chain homotopy. The analogous statement for the cap product involving compactly supported cohomology also holds.
\end{proposition}
\begin{proof}
	As we will prove this generator by generator, the proof will not depend on whether or not we work with compactly supported cohomology.
	
	We first show that we can reduce to the case where $u$ and $w$ are consecutive in a total ordering of the vertices of a finite subcomplex. To see this, observe that the subcomplex spanned by those simplices containing $u$ or $v$ is finite. We can pick a total ordering of the vertices of this subcomplex that restricts to the given orientation on each simplex. The permutations swapping $u$ and $w$ in each simplex are then the composite of finitely many transpositions of consecutive vertices. So, without loss of generality, assume that $u$ and $w$ are consecutive in a total ordering of the vertices of all simplices that contain $u$ or $w$.
	
	Now we claim that if $u$ and $w$ are consecutive, then
	\begin{align}\begin{split}\label{Bchainhty}
		B\colon& C_k(X)\otimes C^l(X) \rar \tilde{C}_{k-l+1}(X)	\\
		&\smallsimplex\otimes\smallsimplexone^* \mapsto \begin{cases}
			(-1)^{k-l}\smallsimplexone^*(\smallsimplex_{\geq k-l})\smallsimplex_{\leq k-l+1}& \mbox{if }\smallsimplex_{k-l}=u\mbox{ and }\smallsimplex_{k-l+1}=w\\ 0& \mbox{otherwise.}
		\end{cases}
\end{split}
	\end{align}
	gives the chain homotopy we need. So, we need to show that for every $\smallsimplex \in X_k$ and $\smallsimplexone \in X_l$ we have
	\begin{equation}
		\sg{\smallsimplex_{\leq k-l}}\smallsimplexone^*(\smallsimplex_{\geq k-l})\smallsimplex_{\leq k-l}-\sg{\smallsimplex}\sg{\smallsimplexone} \smallsimplexone^*(\smallsimplex_{\succcurlyeq k-l}) \smallsimplex_{\preccurlyeq k-l}= \tilde{d}B(\smallsimplex\otimes\smallsimplexone^*)+B(d(\smallsimplex\otimes\smallsimplexone^*)), \label{orderingpropgoal}
	\end{equation}
	where $\preccurlyeq$, $\succcurlyeq $ and $\tilde{d}$ respectively denote the front and back face, and differential with respect to the ordering with $u$ and $w$ swapped. For the right hand side, notice that
	\begin{equation}
		\tilde{d}B(\smallsimplex\otimes\smallsimplexone^*)+B(d(\smallsimplex\otimes \smallsimplexone^*))= \tilde{d}B(\smallsimplex\otimes\smallsimplexone^*)+ B(d\smallsimplex\otimes \smallsimplexone^*)+(-1)^{k-l} B(\smallsimplex\otimes \partial\smallsimplexone^*), \label{chainhtyside}
	\end{equation}
	using the definition of the differential on the tensor product complex $C_k(X)\otimes C^l(X)$.
	
	We start by considering the case where $u\neq \smallsimplex_{k-l}, \smallsimplex_{k-l-1}$ or $w\neq \smallsimplex_{k-l}, \smallsimplex_{k-l+1}$. In this case consecutivity of $u$ and $w$ implies that $\smallsimplex_{\geq k-l}=\smallsimplex_{\succcurlyeq  k-l}$ and $\smallsimplex_{\leq k-l}=\smallsimplex_{\preccurlyeq k-l}$, so both $\cap$ and $\tilde{\cap}$ are only non-zero when $\smallsimplexone= \smallsimplex_{\geq k-l}$. Also by consecutivity of $u$ and $w$ the permutation induced on the vertices of $\smallsimplex$ either acts non-trivially on $\smallsimplex_{\geq k-l}$ or on $\smallsimplex_{\leq k-l}$, and using this it is easy to see that the left hand side of Equation~\eqref{orderingpropgoal} vanishes. For the right hand side (Equation~\eqref{chainhtyside}), we have assumed  $u\neq \smallsimplex_{k-l}$ or $w\neq\smallsimplex_{k-l+1}$ so $B$ vanishes on $\smallsimplex\otimes\smallsimplexone^*$. For the second term our assumptions on $u$ and $w$ imply that none of the faces $\smallsimplex_{\langle i \rangle}$ appearing in the differential of $\smallsimplex$ can have $(\smallsimplex_{\langle i \rangle})_{k-l}=u$ and $(\smallsimplex_{\langle i \rangle})_{k-l+1}=w$ so this term also vanishes. The third term evaluates as
	\begin{equation}
		(-1)^{k-l}B(\smallsimplex\otimes \partial\smallsimplexone^*)= \begin{cases}\begin{matrix}
				(\smallsimplexone^*(\smallsimplex_{\succcurlyeq k-l})-\smallsimplexone^*(\smallsimplex_{\geq k-l}))\smallsimplex_{\leq k-l}	\\+	\sum_{i=2}^{l+1}(-1)^{i+1}\smallsimplexone^*((\smallsimplex_{\geq k-l-1})_{\langle i \rangle}) \smallsimplex_{\leq k-l} 
			\end{matrix}
			& \mbox{if }\begin{matrix}
				\smallsimplex_{k-l-1}=u\\ \mbox{ and }\smallsimplex_{k-l}=w
			\end{matrix}\\ 0& \mbox{otherwise.}
		\end{cases}\label{thirdterm}
	\end{equation}
	which is zero because $u\neq \smallsimplex_{k-l-1}$ or $w\neq \smallsimplex_{k-l}$. 
	
	The remaining two cases are where $u=\smallsimplex_{k-l-1}$ and $w=\smallsimplex_{k-l}$, or where $u=\smallsimplex_{k-l}$ and $w=\smallsimplex_{k-l+1}$. Let us start by working out the right hand side of Equation~\eqref{chainhtyside}. The first term is
	\begin{equation}
		\tilde{d} B(\smallsimplex\otimes\smallsimplexone^*)=\begin{cases}
			\begin{matrix}
				\sum_{i=0}^{k-l-1}(-1)^{i+k-l}	\smallsimplexone^*(\smallsimplex_{\geq k-l})(\smallsimplex_{\geq k-l+1})_{\langle i \rangle}\\ + \smallsimplexone^*(\smallsimplex_{\geq k-l})( \smallsimplex_{\leq k-l}-\smallsimplex_{\preccurlyeq k-l} )
			\end{matrix}& \mbox{if }\smallsimplex_{k-l}=u\mbox{ and }\smallsimplex_{k-l+1}=w\\ 0& \mbox{otherwise.}
		\end{cases}\label{dB}
	\end{equation}
	We recognise the last two terms in this as $\smallsimplex\cap\smallsimplexone^*-\smallsimplex\tilde{\cap}\smallsimplexone^*$. For the second term on the right hand side of Equation~\eqref{chainhtyside} we find that
	$$
	(-1)^{k-l}B(d\smallsimplex\otimes\smallsimplexone^*)=\begin{cases}
		\sum_{i=0}^{k-l-1}(-1)^{i+k-l-1}\smallsimplexone^*(\smallsimplex_{\geq k-l})(\smallsimplex_{\leq k-l+1})_{\langle i \rangle}	&\mbox{if }\begin{matrix}
			\smallsimplex_{k-l}=u\\ \mbox{ and }\smallsimplex_{k-l+1}=w
		\end{matrix}\\
		\sum_{i=k-l+1}^{k}(-1)^{i+k-l-1}\smallsimplexone^*((\smallsimplex_{\langle i \rangle})_{\geq k-l-1})\smallsimplex_{\leq k-l}&	\mbox{if }\begin{matrix}
			\smallsimplex_{k-l-1}=u\\ \mbox{ and }\smallsimplex_{k-l}=w
		\end{matrix}\\
		\smallsimplexone^*((\smallsimplex_{\langle k-l \rangle})_{\geq k-l-1})(\smallsimplex_{\langle k-l \rangle})_{\leq k-l}	&\mbox{if }\begin{matrix}
			\smallsimplex_{k-l-1}=u\\ \mbox{ and }\smallsimplex_{k-l+1}=w.
		\end{matrix}
	\end{cases}
	$$
	The last of these cases cannot occur as $u$ and $w$ are assumed to be consecutive. In the first case we have $(\smallsimplex_{\langle i \rangle})_{\geq k}=\smallsimplex_{\geq k+1}$, and we see that this matches the terms in Equation~\eqref{dB} for $i\geq 2$ with opposite sign. Adding these together now gives the desired result.
\end{proof}

The above is readily adapted to for the cap products from Equations~\eqref{cupversion1}~and~\eqref{cupversion2}:

\begin{proof}[Proof of Proposition~\ref{orindepprop}]
	We only specify what chain homotopies to use in place of Equation~\eqref{Bchainhty}, the rest of the arguments are then a straightforward adaptation of the proof of Proposition~\ref{Rcoefordindep}. For the cap products from Equation~\eqref{cupversion2} mapping
	\begin{align*}
		\cap \colon&	C^\lf_k(X;h^*)\otimes C_c^l(X) \rar C_{k-l}(X;h^*)\mbox{ and}\\
		\cap \colon& 	C^\lf_k(X;h^*)\otimes C^l(X) \rar C^\lf_{k-l}(X;h^*),
	\end{align*}
	we use the chain homotopy
	\begin{align*}
		B\colon& C^\lf_k(X;h^*)\otimes C^l(X) \rar \tilde{C}^\lf_{k-l+1}(X;h^*)	\\
		&(\smallsimplex\under \bigsimplex^*)\otimes\smallsimplexone^* \mapsto \begin{cases}
			(-1)^{k-l}\smallsimplexone^*(\smallsimplex_{\geq k-l})\smallsimplex_{\leq k-l+1}\under \bigsimplex^*& \mbox{if }\smallsimplex_{k-l}=u\mbox{ and }\smallsimplex_{k-l+1}=w\\ 0& \mbox{otherwise,}
		\end{cases}
	\end{align*}
	and the corresponding version for compactly supported cohomology. For the other cap products
	\begin{align*}
		\cap \colon&	C^\lf_k(X;h^*)\otimes C_c^l(X;h_*) \rar C_{k-l}(X)\mbox{ and}\\
		\cap \colon& 	C^\lf_k(X;h^*)\otimes C^l(X;h_*) \rar C^\lf_{k-l}(X),
	\end{align*}
	from Equation~\eqref{cupversion1} we use
	\begin{align*}
		B\colon& C^\lf_k(X;h^*)\otimes C^l(X) \rar \tilde{C}^\lf_{k-l+1}(X;h^*)	\\
		&(\smallsimplex\under \bigsimplex*)\otimes(\smallsimplexone\under \bigsimplexone^*) \mapsto \begin{cases}
			(-1)^{k-l}\langle \bigsimplex^*,\bigsimplexone\rangle\smallsimplex_{\leq k-l+1}& \mbox{if }\smallsimplex_{k-l}=u\mbox{, }\smallsimplex_{k-l+1}=w\\& \mbox{and }\smallsimplexone=\smallsimplex_{\geq k-l} \\ 0& \mbox{otherwise,}
		\end{cases}
	\end{align*}
	and the analogous version for compactly supported cochains.
\end{proof}

\end{document}